\documentclass[11pt,a4paper]{report}
\usepackage{warwickthesis}
\usepackage{setspace} %! setspace is used to control linepacing
\usepackage[square]{natbib} %! needed for Harvard style of references
\usepackage{
amsmath,
amsthm,
amssymb,
amsfonts,
datetime,
enumerate,
graphicx,
mathtools,
listings,
verbatim
}

\usepackage{tikz-cd}
\usepackage{hyperref}

\newtheorem{theorem}{Theorem}[chapter]
\newtheorem{lemma}[theorem]{Lemma}
\newtheorem{corollary}[theorem]{Corollary}
\newtheorem{proposition}[theorem]{Proposition}
\theoremstyle{definition}
\newtheorem{definition}[theorem]{Definition}

 \leftchapter                       %% Uncomment one of these if you want
\newcommand{\mydegree}{MSc}

\renewcommand{\thesistitle}{Standard Representations of Crystallographic Coxeter Groups}       %% The title of your thesis; use mixed-case.

\renewcommand{\thesisauthor} {Emma Booth}    %% Your name
\renewcommand{\thesisauthorpreviousdegrees}{....}  %% Your previous degrees, abbreviated; separate multiple degrees by commas; comment out if not required

\renewcommand{\thesisdegree}{%
  \ifthenelse{\equal{\mydegree}{MSc}}{Master of Science in Mathematics}{%
  \ifthenelse{\equal{\mydegree}{MScY1}}{Summer project for 2y MSc}{%
  {Doctor of Philosophy in Mathematics}
}}}

\renewcommand{\thesissubmission}{Submitted to the University of Warwick\\for the degree of}

\renewcommand{\thesisdepartmentname}{Warwick Mathematics Institute}

\renewcommand{\thesismonth}{\monthname}
\renewcommand{\thesisyear}{\the\year}

\begin{document}

%% Generate the title page
\thesistitlepage

% Roman page numbering for contents, etc
\pagenumbering{roman}

% Doctoral College Guidance July 2024, Sec 2.5 and 2.6:
% Table of contents:
% The table of contents shall immediately follow the title page. It shall list in sequence, with page numbers, all relevant subdivisions of the thesis, including the titles of chapters, sections and subsections, as appropriate; the list of references; the bibliography; the list of abbreviations and other functional parts of the whole thesis together with any appendices. The table of contents should be followed by the list of illustrations and tables.
% If a thesis comprises more than one volume, the contents of the whole thesis shall be shown in the first volume and the contents of each subsequent volume in a separate contents list within that volume.
% Tables and Illustrated Material:
% The lists of tables and illustrations shall follow the table of contents but be placed before the acknowledgements and should include all tables, photographs, diagrams, etc., in the order in which they occur in the text.

\tableofcontents

\pagenumbering{arabic}

\chapter{Introduction}

In the existing literature on reflection groups, there are two distinct definitions of what it means for a Coxeter group $W$ to be crystallographic. These are:
\begin{itemize}
    \item $W$ arises as the Weyl group of a Kac-Moody algebra.
    \item The reflection representation of $W$ preserves a lattice. 
\end{itemize}

These definitions are not in fact equivalent. It is known that $W$ arises as the Weyl group of a Kac-Moody algebra if and only if the edge labels of its Coxeter graph are contained in $\left\{2,3,4,6,\infty\right\}$, and that this is a necessary condition for its reflection representation to preserve a lattice. 

However, there exist cases in which $W$ arises as the Weyl group of a Kac-Moody algebra but its reflection representation fails to preserve a lattice (see Chapter \ref{chap:3}). The Kac-Moody algebra gives rise to a Cartan representation of $W$, and this will always preserve a lattice, so in cases where the reflection representation does not these representations cannot be isomorphic. 

Further, there exist non-isomorphic Kac-Moody algebras with the same Weyl group $W$. The natural question is then whether these give rise to isomorphic Cartan representations. And given a Cartan representation, we can also investigate the question of whether its dual is also isomorphic to some Cartan representation, or when a Cartan representation is isomorphic to its own dual. 

Chapter \ref{chap:2} sets out the results and definitions from the literature that will be used throughout. Chapter \ref{chap:3} explores various small examples which illustrate key points. Chapter \ref{chap:4} sets out necessary and sufficient conditions on a Kac-Moody algebra for the reflection representation and Cartan representations of $W$ to be isomorphic. Chapter \ref{chap:5} classifies the Cartan representations of a given Coxeter group. Chapter \ref{chap:6} considers when a Cartan representation is isomorphic to its dual, or equivalently when it admits a non-degenerate $W$-invariant bilinear form. In Chapter \ref{chap:7} we consider the case when the representations are taken over fields of finite characteristic. And finally in Chapter \ref{chap:8} we summarise the key results proved. 

The primary source used is Carter's \emph{Lie Algebras of Finite and Affine Type} (\cite{Carter}). This sets out clearly how the Cartan representations and their Dynkin diagrams are defined and establishes useful results involving symmetrisability. This is supplemented in places by Kac's \emph{Infinite dimensional Lie algebras} (\cite{Kac}). 

We also make use of Humphreys' \emph{Reflection Groups and Coxeter Groups} (\cite{Humphreys}) for results on crystallographic Coxeter groups. And Vinberg's paper \emph{Discrete Linear Groups Generated By Reflections} (\cite{Vinberg}) serves as inspiration for key results as well as providing two useful definitions and their connection to each other. 

Chapter \ref{chap:2} consists almost entirely of definitions and theorems reproduced from the literature, as cited there. The only original element is the second proof of \ref{lem:symmcond}. Chapter \ref{chap:3} consists entirely of examples. 

The remaining chapters are largely original; we here note the exceptions to that. The main results of Chapters \ref{chap:4} and \ref{chap:5} are special cases of Vinberg's statement (\cite{Vinberg}) that "the Cartan matrix of [a representation of a Coxeter group generated by reflections] is determined up to transformation by a diagonal matrix". However the proofs given here are different and original.

In Chapter \ref{chap:6}, the definition of the dual representation and the following proposition (\ref{prop:isomorphiffform}) are commonly known. It also relies on theory of Kac-Moody algebras found in the literature (\cite{Carter}). And Chapter \ref{chap:7} is entirely original.

\chapter{Preliminaries}

\label{chap:2}

This dissertation relies on knowledge of the theory of Coxeter groups and of Kac-Moody algebras. We summarise the definitions and results needed, and reproduce proofs where appropriate, in this chapter. 
\section{Crystallographic Coxeter Groups}
\begin{definition} (\cite{Humphreys}, p.135) Let $W$ be a Coxeter group and $(V, \rho)$ its reflection representation. Then we say that $W$ is crystallographic if there exists a lattice $L$ in $V$ (i.e. a $\mathbb{Z}$-linear span of a basis of $V$) which is preserved by the action of $W$. \end{definition}

\begin{lemma}(\cite{Humphreys}, p.38) Let $W$ be a crystallographic Coxeter group and $g \in W$. Then $\mathrm{tr}(\rho(g))$ is an integer.\end{lemma}

\begin{proof}Choose a lattice $L$ in $V$ which is preserved by the action of $W$. This lattice contains a basis of $V$, and with respect to this basis the matrix of $\rho(g)$ must have integer entries. So computing the trace in this basis, it must be an integer.\end{proof}

\begin{proposition} (\cite{Humphreys}, p.38, 135) Let $W$ be a crystallographic Coxeter group. Then the edge labels of the Coxeter graph of $W$ are all in $\left\{2,3,4,6,\infty\right\}$.\end{proposition}

\begin{proof}Let $i, j$ be vertices in the Coxeter graph, $m_{ij}$ the label of the edge between them, $s_i, s_j$ the corresponding generators of $W$ and $e_i, e_j$ the corresponding vectors in the usual basis of $V$. Suppose $\dim V = n$. If $m_{ij} = \infty$ the result holds, so suppose otherwise.

Then $\rho(s_i s_j)$ is a rotation by an angle of $\frac{2 \pi}{m_{ij}}$ in the plane spanned by $e_i, e_j$. So its trace is $n - 2 + 2 \cos(\frac{2 \pi}{m_{ij}})$. For this to be an integer (and since $\frac{2 \pi}{m_{ij}} \in (0, \pi]$) we need $m_{ij} = 2,3,4$ or $6$ as required.\end{proof}

We note for future reference the values of $m_{ij}$ and of $b_{ij} = -2(e_i, e_j)$, where $(e_i, e_j) = -1$ if $m_{ij} = \infty$ and $-\cos \frac{\pi}{m_{ij}}$ otherwise is the standard inner product on $V$, in the below table:

\begin{center}
\begin{tabular}{|c|c| }
 \hline
 $m_{ij}$ & $b_{ij}$\\
 \hline
 $2$ & $0$\\
 \hline
 $3$ & $1$\\
 \hline
 $4$ & $\sqrt{2}$\\
 \hline
 $6$ & $\sqrt{3}$\\
 \hline
 $\infty$ & $2$\\
 \hline
\end{tabular}
\end{center}

This condition is not sufficient for $W$ to be crystallographic (see Chapter \ref{chap:3}). We do however have a necessary and sufficient criterion for a Coxeter group to be crystallographic:

\begin{theorem} \label{thm:crystalcond} (\cite{Humphreys}, p.136) Let $W$ be a Coxeter group. Then $W$ is crystallographic if and only if i) all edge labels of the Coxeter graph of $W$ are in $\left\{2,3,4,6,\infty\right\}$, and ii) for every circuit in the Coxeter graph of $W$, the number of edges labelled $4$ (respectively $6$) is even.\end{theorem}

\begin{proof} (\cite{Humphreys}, pp.135-6)

We have already shown that i) is a necessary condition for a Coxeter group to be crystallographic. For ii), note that it is sufficient to prove the statement for cycles. Suppose we have a cycle in the Coxeter graph with vertices labelled $1, \ldots, k$. Let $m_{ij}$ be the edge labels, $s_1, \ldots, s_k$ be the corresponding reflections and $e_1, \ldots, e_k$ the corresponding basis vectors for $V$. 

Let $g = s_1\ldots s_k$. Note that the $e_i$-component of $g \cdot e_i$ is just $1$ for $i \notin \left\{1, \ldots, k\right\}$ so to determine the trace of $\rho(g)$ it is sufficient to consider the action of $g$ on $e_1, \ldots, e_k$. Note that \begin{equation*}s_i(e_j) = e_j - b_{ij} e_i\end{equation*}

The coefficient of $e_1$ in $g \cdot e_1$ is, by direct computation, \begin{equation*}c = b_{12}^2 + b_{k1}^2 -1 + b_{12} b_{23} \ldots b_{k-1, k} b_{k1}\end{equation*} For $i \in \left\{2, \ldots, k-1\right\}$, the coefficient of $e_i$ in $g \cdot e_i$ is \begin{equation*}b_{i,i+1}^2 - 1\end{equation*} And the coefficient of $e_k$ in $g \cdot e_k$ is $-1$. 

The squares are all integers, so for the trace of $\rho(g)$ to be an integer we need $b_{12} b_{23} \ldots b_{k-1, k} b_{k1}$ to be an integer. And since the $b_{ij}$ take values in $\left\{0,1, \sqrt{2}, \sqrt{3}, 2\right\}$ this occurs if and only if there are an even number of $\sqrt{2}$s and $\sqrt{3}$s if and only if there are an even number of edges labelled $4$ and $6$ in the cycle. 

Suppose conversely that the hypotheses i) and ii) hold. Let $m_{ij}$ be the edge labels. It is sufficient to choose some real $\lambda_1, \ldots \lambda_n$ such that $\lambda_i = \lambda_j$ if $m_{ij}$ is $3$ or $\infty$, $\lambda_i = (\sqrt{2})^{\pm 1} \lambda_j$ if $m_{ij}$ is $4$ and $\lambda_i = (\sqrt{3})^{\pm 1} \lambda_j$ if $m_{ij}$ is $6$. If we can do this, then $\lambda_i e_i$ span a lattice which is invariant under the action of $W$.

So it suffices to prove that we can make a well-defined choice of the $\lambda_i$, which we do as follows: without loss of generality the Coxter graph $\Gamma$ of $W$ is connected. Fix a vertex $i$ and define $\lambda_i = 1$. For each $j$, choose a path from $i$ to $j$ and define values of $\lambda_{k_l}$ along this path. If $m_{k_l k_{l+1}}$ is $3$ or $\infty$ then let $\lambda_{k_{l+1}} = \lambda_{k_l}$. Where we encounter edges of label $4$ (respectively $6$), alternate multiplying and dividing by $\sqrt{2}$ (respectively $\sqrt{3}$). This will eventually define $\lambda_j$.

The only possible ambiguity is if we have two paths from $i$ to $j$, which together give a circuit in $\Gamma$. The assumption that the numbers of edges labelled $4$ and $6$ in this cycle are even then means that $\lambda_j$ is well-defined. It follows that the required conditions are satisfied for all edges. 
\end{proof}
\section{Weyl Groups, Dynkin Diagrams and Cartan Representations}

\begin{definition} (\cite{Carter}, p.319) Let $A$ be an $n \times n$ matrix with entries in $\mathbb{Z}$. We say that $A$ is a generalised Cartan matrix if it satisfies the conditions: 
\begin{itemize}
 \item $A_{ii} = 2$ for every $i$,
 \item $A_{ij} \leq 0$ for $i \neq j$, and,
 \item $A_{ij} = 0$ if and only if $A_{ji} = 0$.
 \end{itemize}
\end{definition}
In this case, we can associate a Lie algebra (in general infinite-dimensional, and over an arbitrary field $k$) to $A$, which we call the Kac-Moody algebra $L(A)$ given by $A$. These generalise the theory of finite-dimensional semi-simple Lie algebras, and as in that theory we obtain a root system for $L(A)$. 

\begin{definition} (\cite{Carter}, pp.319-320) Let $A$ be an $n \times n$ generalised Cartan matrix of rank $l$. Then a minimal realisation of $A$ is a triple $(\mathfrak{h}, \Pi, \Pi^\wedge)$ such that $\mathfrak{h}$ is a complex vector space of dimension $2n - l$, $\Pi = \left\{\alpha_1, \ldots, \alpha_n\right\}$ and $\Pi^\wedge = \left\{h_1, \ldots, h_n\right\}$ are linearly independent subsets of $\mathfrak{h}^*, \mathfrak{h}$ respectively, and $\alpha_j(h_i) = A_{ij}$ for all $i, j$.\end{definition}

It can be shown that a minimal realisation of $A$ exists, and that while it is not necessarily unique the choices made in its definition are arbitrary (see e.g. Chapter Sixteen of \cite{Carter} and Chapter One of \cite{Kac}). $\mathfrak{h}$ is called a Cartan subalgebra of the Kac-Moody algebra $L(A)$, and in the case that $L(A)$ is a finite-dimensional semisimple Lie algebra $\mathfrak{h}$ is indeed a Cartan subalgebra in that sense. 

 \begin{definition} \label{def:simpleroots} $\alpha_1, \ldots, \alpha_n$ are called the simple roots of $L(A)$, and form a base for its root system. Let $H$ be the free $\mathbb{Z}$-module spanned by $\alpha_1, \ldots, \alpha_n$. \end{definition}

 \begin{definition}(\cite{Kac}, p.35) The $i$th fundamental reflection is the $\mathbb{Z}$-linear map $s_i: H \to H$ defined by $s_i(\alpha_j) = \alpha_j - A_{ij} \alpha_i$.\end{definition}

\begin{definition} (\cite{Carter}, p.373) The subgroup of $\mathrm{GL}(H)$ (i.e. the group of invertible $\mathbb{Z}$-linear maps $H \to H$) generated by the fundamental reflections is called the Weyl group $W$ of $L(A)$. \end{definition}

\begin{theorem} \label{thm:weyliscoxeter} (Theorem 16.17 of \cite{Carter}, p.376) The Weyl group $W$ of the Kac-Moody algebra $L(A)$ is the Coxeter group generated by $s_1, \ldots, s_n$ with relations:
\begin{itemize}
\item$s_i^2 = 1$,
\item$(s_i s_j)^2 = 1$ if $A_{ij} A_{ji} = 0$,
\item $(s_i s_j)^3 = 1$ if $A_{ij} A_{ji} = 1$,
\item $(s_i s_j)^4 = 1$ if $A_{ij} A_{ji} = 2$, and,
\item $(s_i s_j)^6 = 1$ if $A_{ij} A_{ji} = 3$.\end{itemize}\end{theorem}

\begin{definition}If $k$ is a field (of any characteristic) we can regard $k$ as a $\mathbb{Z}$-module in the natural way. Define $H_k = H \bigotimes k$ as a tensor product of $\mathbb{Z}$-modules. This is a $k$-vector space of dimension $n$.\end{definition}

\begin{definition} Let $W$ be a Coxeter group which emerges as the Weyl group of a Kac-Moody algebra $L(A)$. Then $H_k$ together with the natural action of the Weyl group is called a Cartan representation of $W$ over $k$.\end{definition}

If we suppose that $L(A)$ is taken to be over the field $k$, then $H_k$ is a subspace of the dual $\mathfrak{h}^*$ of the Cartan subalgebra of $L(A)$. In fact, since $H_k$ has dimension $n$ and $\mathfrak{h}^*$ has dimension $2n - \text{rank}(A)$, $H_k = \mathfrak{h}^*$ if and only if $\det A \neq 0$.

Note that in the case $k = \mathbb{R}$, the action of $W$ preserves the lattice spanned by $\alpha_1, \ldots, \alpha_n$. In particular, if $H_\mathbb{R}$ is isomorphic to the reflection representation of $W$, $W$ must be crystallographic. 

As in the finite-dimensional case, we can classify the root systems of Kac-Moody algebras according to their Dynkin diagrams. It is convenient to associate a Dynkin diagram directly with the generalised Cartan matrix. 

\begin{definition} \label{def:dynkin} (\cite{Carter}, p.353 and \cite{Kac}, p.51; the exact notation varies slightly between texts) Let $A$ be an $n \times n$ generalised Cartan matrix. Then the Dynkin diagram $\Delta(A)$ of $A$ is the diagram with vertices given by $\left\{1,\ldots,n\right\}$ and edges between $i$ and $j$ determined by the values of $A_{ij}$ and $A_{ji}$ as follows:

\begin{itemize}
\item If $A_{ij} A_{ji} = 0$ there is no edge between $i$ and $j$.

\item If $A_{ij} A_{ji} = 1$ there is a single undirected edge between $i$ and $j$.

\item If $A_{ij} = -1, A_{ji} = -2$ then $i$ and $j$ are joined by a double edge directed towards $j$; similarly if $A_{ji} = -1, A_{ij} = -2$ then $i$ and $j$ are joined by a double edge directed towards $i$.

\item If $A_{ij} = -1, A_{ji} = -3$ then $i$ and $j$ are joined by a triple edge directed towards $j$; similarly if $A_{ji} = -1, A_{ij} = -3$ then $i$ and $j$ are joined by a triple edge directed towards $i$.

\item If $A_{ij} = A_{ji} = -2$ then $i$ and $j$ are joined by an undirected quadruple edge. 

\item If $A_{ij} = -1, A_{ji} = -4$ then $i$ and $j$ are joined by a quadruple edge directed towards $j$. Similarly if $A_{ij} = -4, A_{ji} = -1$.

\item If $A_{ij} A_{ji} \geq 5$, $i$ and $j$ are joined by a single bold-faced edge labelled by the  pair $(|A_{ij}|, |A_{ji}|)$ and directed towards $i$ if $A_{ij} > A_{ji}$ and $j$ if $A_{ji} > A_{ij}$. (The edge is not directed if $A_{ij} = A_{ji}$.)
\end{itemize}
\end{definition}
Several examples of Dynkin diagrams will follow in Chapter \ref{chap:3}.

\begin{theorem} Let $W$ be a Coxeter group with Coxeter graph $\Gamma$, which arises as the Weyl group of a Kac-Moody algebra $L(A)$, and $\Delta$ the Dynkin diagram of $A$. Associate the vertices of $\Gamma$ to the vertices of $\Delta$ via the natural bijection. Then:
\begin{itemize}
\item There is no edge between $i$ and $j$ in $\Gamma$ if and only if there is no edge between $i$ and $j$ in $\Delta$.

\item There is an edge labelled $3$ between $i$ and $j$ in $\Gamma$ if and only if there is a single unlabelled edge between $i$ and $j$ in $\Delta$.

\item There is an edge labelled $4$ between $i$ and $j$ in $\Gamma$ if and only if there is a double edge between $i$ and $j$ in $\Delta$.

\item There is an edge labelled $6$ between $i$ and $j$ in $\Gamma$ if and only if there is a triple edge between $i$ and $j$ in $\Delta$.

\item There is an edge labelled $\infty$ between $i$ and $j$ in $\Gamma$ if and only if there is either a quadruple edge or a labelled edge between $i$ and $j$ in $\Delta$.
\end{itemize} 
\end{theorem}

\begin{proof}This follows directly from Theorem \ref{thm:weyliscoxeter}, Definition \ref{def:dynkin} and the definition of a Coxeter graph.\end{proof}

In particular, we observe that the Dynkin diagram uniquely determines the Coxeter graph, but we may in general associate several distinct Dynkin diagrams to a Coxeter graph: we have a choice of direction in the cases of edges labelled $4$ and $6$, and there are infinitely many possibilities where an edge is labelled $\infty$.

From the perspective of groups, this means that every Kac-Moody algebra has a unique Weyl group, but given a Coxeter group $W$ there may be many distinct Kac-Moody algebras with $W$ for a Weyl group and thus many Cartan representations. Any $W$ with edge labels in $\left\{2,3,4,6,\infty\right\}$ will arise as the Weyl group of at least one Kac-Moody algebra and so have at least one Cartan representation. 

Natural questions then arise: when are two Cartan representations of a given Coxeter group $W$ isomorphic? When are they isomorphic to the reflection representation of $W$? We will aim to answer these questions in subsequent chapters, but first need some more definitions. 

\begin{definition} (\cite{Kac}, p.2, \cite{Carter}, p.336) Let $A$ be an $n \times n$ generalised Cartan matrix. Then $A$ is decomposable if there exists a partition of $\left\{1, \ldots, n\right\}$ into two non-empty sets $I, J$ such that $A_{ij} = A_{ji} = 0$ for all $i \in I, j \in J$. If $A$ is not decomposable, we say that $A$ is indecomposable.\end{definition}

\begin{proposition} \label{prop:indecompcond} Let $A$ be a generalised Cartan matrix. Then the following are equivalent:

\begin{itemize}
    
\item $A$ is indecomposable.

\item The Dynkin diagram of $A$ is connected.

\item The Cartan representation of the Weyl group of the associated Kac-Moody algebra (over any field of characteristic $0$) is indecomposable. 
\end{itemize}
\end{proposition}
The proof is straightforward verification, and is omitted. This means that we can without loss of generality only study indecomposable Cartan matrices (at least in characteristic $0$), and will do so in this text. 

\begin{definition} (\cite{Vinberg}) Let $A$ be an $n \times n$ generalised Cartan matrix. Then the cyclic products of $A$ are the quantities $A_{i_1 i_2} \ldots A_{i_k i_1}$ for $i_1, \ldots, i_k \in \left\{1,\ldots,n\right\}$ all distinct.\end{definition}

\begin{definition} (\cite{Vinberg}) Two generalised Cartan matrices $A, B$ are called equivalent if there exists a diagonal matrix $D$ with positive diagonal entries such that $A = DBD^{-1}$.\end{definition}

Note that this definition is not universal. Often, a different notion of equivalence is used: $A$ and $B$ are equivalent if there exists a permutation $\sigma$ of $1, \ldots, n$ such that $B_{ij} = A_{\sigma(i) \sigma(j)}$ for all $i, j$ (\cite{Carter}, p.336). But this permuting of the entries is less useful for these purposes as it does not preserve the Cartan representation.

\begin{proposition} \label{prop:equivcond} Generalised Cartan matrices $A, B$ are equivalent if and only if they have the same cyclic products. \end{proposition}

\begin{proof}(\cite{Vinberg}) Suppose $A, B$ are equivalent and write $A = DBD^{-1}$ as above. Then \begin{equation*}A_{ij} = \Sigma_{k, l} D_{ik} B_{kl} D^{-1}_{lj} = d_i d_j^{-1} B_{ij}\end{equation*} Thus \begin{equation*}A_{i_1i_2} \ldots A_{i_k i_1} = d_{i_1} d_{i_2}^{-1} \ldots d_{i_k} d_{i_1}^{-1} B_{i_1 i_2} \ldots B_{i_k i_1} = B_{i_1 i_2} \ldots B_{i_k i_1}\end{equation*} So the cyclic products are equal.

Suppose conversely all cyclic products are equal. If $A_{ij} = 0$ then $A_{ij} A_{ji} = 0$, which means $B_{ij} B_{ji} = 0$ also and so $B_{ij} = 0$. Similarly if $B_{ij} = 0$ then $A_{ij} = 0$. If $A_{ij} \neq 0$, let $c_{ij} = A_{ij} B_{ij}^{-1}$. Then $c_{ij} > 0$ and \begin{equation*}c_{i_1i_2} \ldots c_{i_k i_1} = 1\end{equation*} whenever all factors on the left-hand side are defined. 

This means we can make a well-defined choice of $d_1, \ldots, d_n > 0$ such that $c_{ij} = d_i d_j^{-1}$ for each $i, j$: let $d_1 = 1$ (or indeed any positive rational number). For each $i$, choose a path $1 = j_1, \ldots, j_k = i$ from $1$ to $i$ in the Dynkin diagram $\Delta$ of $A$ (or equivalently of $B$, as there is an edge from $i$ to $j$ in $\Delta(A)$ if and only if there is such a path in $\Delta(B)$). Define $d_{j_2} = c_{j_2j_1} d_{j_1}$, $d_{j_3} = c_{j_3j_2} d_{j_2}\ldots$, and repeat until $d_i$ is defined. The given condition on the products of the $c_{ij}$ means that this is well-defined, and since the $c_{ij}$ are all positive the $d_i$ also are. 

Now \begin{equation*}A_{ij} = c_{ij} B_{ij} = d_i B_{ij} d_j^{-1}\end{equation*} for all $i, j$. So if $D$ is the matrix $\text{diag}(d_1, \ldots, d_n)$ then $A = DBD^{-1}$ as required. \end{proof}

\begin{definition} (\cite{Carter}, p.346) A generalised Cartan matrix $A$ is symmetrisable if there exists a non-singular diagonal matrix $D$ and a symmetric matrix $B$ such that $A = DB$. \end{definition}

It is convenient to work with the following equivalent condition:  

\begin{lemma} \label{lem:symmcond} (Lemma 15.15 of \cite{Carter}, p.346) Let $A$ be an $n \times n$ generalised Cartan matrix. Then $A$ is symmetrisable if and only if \begin{equation*}A_{i_1i_2} A_{i_2 i_3} \ldots A_{i_k i_1} = A_{i_2 i_1} A_{i_3 i_2} \ldots A_{i_1 i_k}\end{equation*} for all $i_1, \ldots, i_k \in \left\{1, \ldots,n\right\}$. That is, the cyclic products of $A$ are the same as the cyclic products of $A^T$, and so $A$ and $A^T$ are equivalent.\end{lemma} 

\begin{proof} (\cite{Carter}, pp.346-347) Suppose $A$ is symmetrisable and write $A = DB$ where $D = \text{diag}(d_1, \ldots, d_n)$ and $B$ is symmetric. Thus $A_{ij} = d_i B_{ij}$ for all $i, j$. Therefore \begin{equation*}A_{i_1 i_2} \ldots A_{i_k i_1} = d_{i_1} \ldots d_{i_k} B_{i_1 i_2}  \ldots B_{i_k i_1}\end{equation*} and \begin{equation*}A_{i_2 i_1} \ldots A_{i_1 i_k} = d_{i_1} \ldots d_{i_k} B_{i_2 i_1} \ldots B_{i_1 i_k}\end{equation*} And these expressions are equal since $B$ is symmetric.

Suppose conversely the given condition holds. Without loss of generality $A$ is indecomposable. For each $i \in \left\{1, \ldots, n\right\}$, we can thus choose a path $1 = j_1, \ldots, j_r = i$ from $1$ to $i$ in the Dynkin diagram of $A$. That is, $j_2, \ldots, j_{r-1}$ with $A_{j_k j_{k+1}} \neq 0$ for each $k$. 

Set $d_1$ to be any non-zero real number. We want to define \begin{equation*}d_i = \dfrac{A_{j_r j_{r-1}}\ldots A_{j_2 j_1}}{A_{j_11j_2} \ldots A_{j_{r-1} j_r}} d_1\end{equation*} for each $i$. For this to be well-defined, we need it to be independent of the choice of path from $1$ to $i$. 

Let $1 = k_1, \ldots, k_s = i$ be a second such path. Then we want \begin{equation*}\dfrac{A_{j_r j_{r-1}}\ldots A_{j_2 j_1}}{A_{j_11j_2} \ldots A_{j_{r-1} j_r}} = \dfrac{A_{k_s k_{s-1}}\ldots A_{k_2 k_1}}{A_{k_11k_2} \ldots A_{k_{r-1} k_r}}\end{equation*} But if we multiply out the denominators we obtain precisely the given condition in the case $i_1, \ldots, i_k = 1= j_1, j_2, \ldots, j_r = i = k_s, k_{s-1}, \ldots, k_2$. So the $d_i$ are well-defined.

Now let $D = \text{diag}(d_1, \ldots, d_n)$ and define $B_{ij} = \frac{1}{d_i} A_{ij}$. Then $A = DB$, with $D$ diagonal and non-singular. So it's sufficient to prove that $B$ is symmetric. That is, \begin{equation*}\dfrac{1}{d_j} A_{ji} = \dfrac{1}{d_i} A_{ij}\end{equation*} If $i = j$ the condition is satisfied trivially, and if $A_{ij} = 0$ then $A_{ji} = 0$ also so the condition is still satisfied. 

Suppose hence $i \neq j, A_{ij} \neq 0$. Take a path $1 = j_1, \ldots, j_r = i$ from $1$ to $i$. Then $j_1, \ldots, j_r, j$ is a path from $1$ to $j$. We can use these paths to obtain $d_i, d_j$ respectively, and doing so gives \begin{equation*}d_j = \frac{A_{ji}}{A_{ij}} d_i\end{equation*} Rearranging this gives the condition for $B$ to be symmetric. So $A$ is symmetrisable. \end{proof}

\begin{proof}[Alternative Proof] Suppose $A$ is symmetrisable and write $A = DB$ where $D = \text{diag}(d_1, \ldots, d_n)$ is diagonal and $B$ is symmetric. Then $A_{ij} = d_i B_{ij}$ and so $d_j = \frac{A_{ji}}{A_{ji}} d_i$. Hence all $d_i$ have the same sign (as $A_{ij}$ and $A_{ji}$ always have the same sign and again assuming indecomposability), and without loss of generality by multiplying through by $-1$ they are all positive. Now $A^T = (DB)^T = BD$. Thus $B = D^{-1} A$ and $B = A^T D^{-1}$, and hence $A = D A^T D^{-1}$. So $A, A^T$ are equivalent.

Suppose conversely $A, A^T$ are equivalent and write $A = D A^T D^{-1}$ where $D$ is diagonal (and non-singular). Let $B = A^T D^{-1}$ so that $A = DB$. Then $B^T = D^{-1} A = D^{-1} D A^T D^{-1} = B$. So $B$ is symmetric, and $A$ is symmetrisable. \end{proof}

We note also that when using the cyclic product definition it is sufficient to only consider the case in which $i_1, \ldots, i_k, i_1$ is a circuit in the Dynkin diagram of $A$, since otherwise there exists some $j$ such that $A_{i_j i_{j+1}} = A_{i_{j+1} i_j}$ and so both sides of the condition are equal to zero.

\begin{definition} (\cite{Kac}, p.30) Let $A$ be a generalised Cartan matrix and $\mathfrak{g}(A)$ its associated Kac-Moody algebra. Then the dual of $\mathfrak{g}(A)$ is the Kac-Moody algebra associated to $A^T$.\end{definition}

Note that $\mathfrak{g}(A)$ and its dual have the same Weyl group (up to isomorphism) as the Weyl group is determined only by the values of $A_{ij} A_{ji} = (A^T)_{ij} (A^T)_{ji}$. This is not necessarily the same as the dual representation of the Cartan  representation. 

Now we have a classification theorem for indecomposable GCMs. Where $u \in \mathbb{R}^n$ is a vector, we will write $u > 0$ to mean that each component of $u$ is positive, and similarly $u \geq 0, u < 0$ or $u \leq 0$.

\begin{theorem} \label{thm:gcmclassification} (\cite{Carter}, pp.336-337) Let $A$ be an indecomposable $n \times n$ generalised Cartan matrix. Then exactly one of the following three cases holds for $A$: 

\begin{itemize}
    \item  $\det A \neq 0$, there exists $u > 0$ with $Au > 0$, and $Au \geq 0$ implies $u > 0$ or $u = 0$.

 \item  $A$ has rank $n-1$, there exists $u > 0$ such that $Au = 0$, and $Au \geq 0$ implies $Au = 0$.

 \item There exists $u > 0$ such that $Au < 0$, and $Au \geq 0, u \geq 0$ imply $u = 0$.
\end{itemize}
\end{theorem}
We say that $A$ has finite type if the first case holds, affine type if the second holds, and indefinite type if the third holds. 

If $A$ has finite type, then the Kac-Moody algebra $\mathfrak{g}(A)$ is in fact a finite-dimensional complex semisimple Lie algebra, but if $A$ has affine or indefinite type then $\mathfrak{g}(A)$ is infinite-dimensional. 

Now it can also be shown (\cite{Carter}, p.350) that a generalised Cartan matrix of finite type has positive determinant. Generalised Cartan matrices of affine type have determinant $0$ by definition. So any generalised Cartan matrix with negative determinant must be of indefinite type. 

However, the converse is false. For instance, the GCM 

\begin{equation*}\begin{pmatrix}
    2 & -3 & 0 & 0\\
    -3 & 2 & -1 & 0\\
    0 & -1 & 2 &-3\\
    0 & 0 & -3 & 2\\
\end{pmatrix}\end{equation*}

has determinant $21$ and the GCM 
\begin{equation*}\begin{pmatrix}
    2 & 0 & -1 & -3\\
    0 & 2 & -3 & -1\\
    -1 & -3 & 2 & 0\\
    -3 & -1 & 0 &2\\
\end{pmatrix}\end{equation*}

has determinant $0$. Further, both of these have indefinite type (as the classification of finite and affine indecomposable GCMs is known).

\chapter{Some Illustrative Examples}

\label{chap:3}

Let's consider some small examples of Cartan subalgebra representations. The smallest non-trivial case is that of a $2 \times 2$ generalised Cartan matrix, $A = \begin{pmatrix}2 & -m\\-n & 2\\\end{pmatrix}$ where $m, n$ are non-negative integers, and $m = 0$ if and only if $n = 0$. In fact in the $m = n = 0$ case the GCM decomposes as two copies of the unique $1 \times 1$ GCM, so we can suppose $m, n$ are positive. 

There are four possible Weyl groups of $\mathfrak{g}(A)$, which are the Coxeter groups with the following graphs:

 \begin{figure}[h]
\centering
\begin{tikzcd}
\bullet  \arrow[r, dash] & \bullet
\end{tikzcd}
\end{figure}
\begin{figure}[h]
\centering
\begin{tikzcd}
\bullet  \arrow[r, dash, "4"] & \bullet
\end{tikzcd}
\end{figure}
\begin{figure}[h]
\centering
\begin{tikzcd}
\bullet  \arrow[r, dash, "6"] & \bullet
\end{tikzcd}
\end{figure}
\begin{figure}[h!]
\centering
\begin{tikzcd}
\bullet  \arrow[r, dash, "\infty"] & \bullet
\end{tikzcd}
\end{figure}

The first three of these are finite, and are isomorphic to the dihedral groups $\text{Dih}(6), \text{Dih}(8)$ and $\text{Dih}(12)$ respectively. The fourth is the "infinite dihedral group" $\text{Dih}(\infty) = <s_1, s_2|s_1^2, s_2^2>$. All four of these groups are crystallographic, by Theorem \ref{thm:crystalcond}.

Let's also list the possible Dynkin diagrams. There's only a single one with $\text{Dih}(6)$ for a Weyl group:

\begin{figure}[h!]
\centering
\begin{tikzcd}
\bullet  \arrow[r, dash] & \bullet
\end{tikzcd}
\end{figure}

corresponding to the case $m = n = 1$. We obtain two cases with $\text{Dih}(8)$ for a Weyl group, 

\begin{figure}[h]
\centering
\begin{tikzcd}
\bullet \arrow[r, Rightarrow] & \bullet
\end{tikzcd}
\end{figure}

\begin{figure}[h]
\centering
\begin{tikzcd}
\bullet  & \bullet \arrow[l, Rightarrow]
\end{tikzcd}
\end{figure}

which correspond to $\begin{pmatrix}2 & -1\\-1 & 2\\\end{pmatrix}$ and its transpose. Similarly for $\text{Dih}(12)$ we have two cases given by a triple edge in each direction, $m = -3, n = -1$ and $m = -1, n = -3$.

All GCMs with $mn \geq 4$ have $\text{Dih}(\infty)$ for a Weyl group. There are infinitely many of these. We'll show the Dynkin diagrams for the case $mn = 4$:

\begin{figure}[h]
\centering
\begin{tikzcd}
\bullet  \arrow[r, Rightarrow, shift right] \arrow[r, Rightarrow, shift left]& \bullet
\end{tikzcd}
\end{figure}
\begin{figure}[h!]
\centering
\begin{tikzcd}
\bullet & \bullet \arrow[l, Rightarrow, shift right] \arrow[l, Rightarrow, shift left] 
\end{tikzcd}
\end{figure}

\begin{figure}[h!]
\centering
\begin{tikzcd}
\bullet  \arrow[r, dash] \arrow[r, dash, shift left] \arrow[r, dash, shift right] \arrow[r, dash, shift left=2] & \bullet
\end{tikzcd}
\end{figure}

Note that these GCMs are all symmetrisable, as none of the Dynkin diagrams contain cycles. So in particular the two GCMs with $mn = 2$ are equivalent by Lemma \ref{lem:symmcond}, as are those with $mn = 3$ and the two directed cases with $mn = 4$. We can check that the non-directed $mn = 4$ case is also equivalent to these, using the diagonal matrix $D = \begin{pmatrix}1 & 0\\0 & 2\\\end{pmatrix}$. 

In fact, more generally, any two $2 \times 2$ generalised Cartan matrices $A, B$ with $A_{12} A_{21} = B_{12} B_{21}$ are equivalent. This is because the only possible cyclic product is $A_{12} A_{21}$. So the cyclic products for $A, B$ are all equal. And conversely if $A_{12} A_{21} \neq B_{12} B_{21}$ then the GCMs cannot be equivalent.

Let $\text{Dih}(6) = <s, t|s^2, t^2, (st)^3>$. Then its action on its reflection representation $R = <e_1, e_2>$ is given by $s \cdot e_1 = - e_1, t \cdot e_2 = - e_2, s \cdot e_2 = e_1 + e_2, t \cdot e_1 = e_1 + e_2$. This is the same action as the action on the Cartan representation, as determined by the GCM $\begin{pmatrix}2 & -1\\-1 & 2\\\end{pmatrix}$.

The case of $\text{Dih}(8) = <s,t|s^2, t^2, (st)^4>$ is more interesting: the action on the reflection representation $R$ is given by $s \cdot e_1 = - e_1, t \cdot e_2 = - e_2, s \cdot e_2 = \sqrt{2} e_1 + e_2, t \cdot e_1 = e_1 + \sqrt{2} e_2$. And there are two Cartan representations $H_1, H_2$ corresponding to the matrices $\begin{pmatrix}2 & -1\\-2 & 2\\\end{pmatrix}$ and $\begin{pmatrix}2 & -2\\-1 & 2\\\end{pmatrix}$.

It turns out that these three representations are all isomorphic: define \begin{equation*}\varphi_1: R \to H_1, \varphi_1(e_1) = \alpha_1, \varphi_1(e_2) = \sqrt{2} \alpha_2\end{equation*} and \begin{equation*}\varphi_2: R \to H_2, \varphi_2(e_1) = \sqrt{2} \alpha_1, \varphi_2(e_2) = e_2\end{equation*} And similarly, the representations of $\text{Dih}(12)$ are isomorphic both to each other and to the reflection representation of $\text{Dih}(12)$.

In fact, we claim that the Cartan representations corresponding to $\begin{pmatrix}2 & -a\\-b & 2\\\end{pmatrix}$ and $\begin{pmatrix}2 & -m\\-n & 2\\\end{pmatrix}$ are isomorphic if and only if $ab = mn$.

\begin{proof} It suffices to prove the case $mn \geq 4$ where the Weyl group is $\text{Dih}(\infty)$. We'll show that $A = \begin{pmatrix}2 & -m\\-n & 2\\\end{pmatrix}$ and $B = \begin{pmatrix}2 & -1\\-mn & 2\\\end{pmatrix}$ give isomorphic representations. Let $H_1, H_2$ be the two representations and have standard bases $\alpha_1, \alpha_2$ and $\beta_1, \beta_2$ respectively, and $s_1, s_2$ generate $\text{Dih}(\infty)$.

Define $\varphi: H_1 \to H_2$ by \begin{equation*}\varphi(\alpha_1) = \beta_1, \varphi(\alpha_2) = m \beta_2\end{equation*} Then \begin{equation*}s_i \cdot \varphi(\alpha_i) = \varphi (s_i \cdot \alpha_i)\end{equation*} for $i = 1, 2$. \begin{equation*}s_1 \cdot \varphi(\alpha_2) = s_1 \cdot m \beta_2 = m \beta_2 + m \beta_1\end{equation*} and \begin{equation*}\varphi(s_1 \cdot \alpha_2) = \varphi(\alpha_2 + m \alpha_1) = m \beta_2 + m \beta_1\end{equation*} 

Finally \begin{equation*} s_2 \cdot \varphi(\alpha_1) = s_2 \cdot \beta_1 = \beta_1 + mn \beta_2\end{equation*} and \begin{equation*} \varphi(s_2 \cdot \alpha_1) = \varphi(\alpha_1 + n\alpha_2) = \beta_1 + mn \beta_2\end{equation*}

So we have the required isomorphism. (Proof of the converse is omitted but will follow from the general results in Chapter \ref{chap:5}).
\end{proof}

Also, the reflection representation of $\text{Dih}(\infty)$ is isomorphic to a Cartan representation in precisely the case $mn = 4$: if it has basis $e_1, e_2$ then $s_i \cdot e_i = -e_i$, $s_1 \cdot e_2 = 2e_1 + e_2$,  $s_2 \cdot e_1 = e_1 + 2e_2$. This is the same action as that of the representation corresponding to $\begin{pmatrix}2 & -2\\-2 & 2\\\end{pmatrix}$. 

More interesting cases arise when we permit the Coxeter graph to contain cycles. For instance, let's consider the Coxeter graph $\Gamma$: 

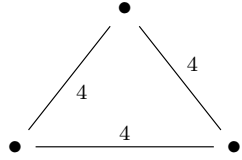
 \begin{figure}[h]
\centering
\begin{tikzcd}
& \bullet \arrow[ddr, dash, "4"] \arrow[ddl, dash, "4"] & \\
& & \\
\bullet \arrow[rr, dash, "4"] & & \bullet\\
\end{tikzcd}
\end{figure}

This is the smallest (in terms of the order and size of the graph, and the edge labels) example of a Coxeter group $W$ which is not crystallographic (by Theorem 2.3, as the cycle in $\Gamma$ contains an odd number of edges of label $4$.) So none of the Cartan representations can be isomorphic to the reflection representation.

There are in fact eight such representations, corresponding to two choices of direction for each of the three double edges in the Dynkin diagram. We reproduce some of them below. Two of the representations have all edges in the same direction. One of them is the below:

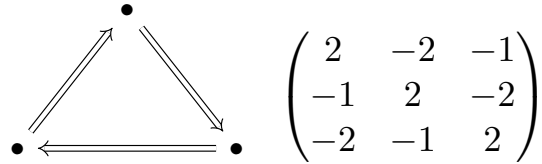
\begin{figure}[h]
\centering

\begin{tikzcd}
    & \bullet \arrow[ddr, Rightarrow] & \\
    & & \\
    \bullet \arrow[uur, Rightarrow] & & \bullet \arrow[ll, Rightarrow]\\
\end{tikzcd}
\scalebox{1.25}{
$\begin{pmatrix}2 & -2 & -1\\-1 & 2 & -2\\-2 & -1 & 2\\\end{pmatrix}$
}
\end{figure}

shown together with its GCM. 

The remaining six representations have two edges in one direction and the last edge in the opposite direction, such as: 

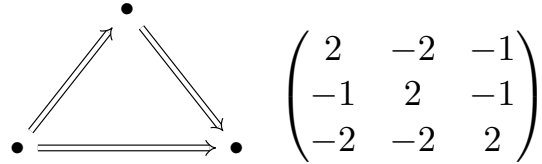
\begin{figure}[h]
\centering
\begin{tikzcd}
    & \bullet \arrow[ddr, Rightarrow] & \\
    & & \\
    \bullet \arrow[uur,Rightarrow] \arrow[rr, Rightarrow] & & \bullet \\
\end{tikzcd}
\scalebox{1.25}{
$\begin{pmatrix}2 & -2 & -1\\-1 & 2 & -1\\-2 & -2 & 2\\\end{pmatrix}$
}
\end{figure}

Fix some labelling of the vertices as $1,2,3$. $A_{ij} A_{ji} = 2$ for all $i, j$ for each representation, so equivalence of the GCMs $A$ and $B$ with this Weyl group is determined only by whether $A_{12} A_{23} A_{31} = B_{12} B_{23} B_{31}$. 

In fact there are four possible values of the cyclic product $A_{12} A_{23} A_{31}$: 

\begin{itemize}
\item $-8$ (three clockwise edges, one representation),
\item $-4$ (two clockwise edges and one anticlockwise edge, three representations),
\item $-2$ (two anticlockwise edges and one clockwise edge, three representations), and,
\item $-1$ (three anticlockwise edges, one representation). 
\end{itemize}

And we can check that these also give the isomorphism classes of the representations, so we have four isomorphism classes corresponding to these values. 

We can transform between the Dynkin diagrams of isomorphic representations by "flipping" edges: where a vertex is a source, with two edges entering it, we can replace it by a sink, with two edges leaving it, and conversely. This corresponds to e.g. dividing $A_{12}$ by $2$ and multiplying $A_{21}$ by $2$, and dividing $A_{13}$ by $2$ and multiplying $A_{31}$ by $2$, which will preserve the cyclic products. 

Let's now consider the example of the Coxeter graph $\Gamma$

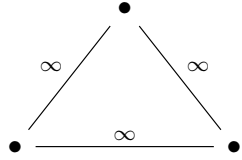
\begin{figure}[h]
\centering
\begin{tikzcd}
    & \bullet \arrow[ddr, dash, "\infty"] & \\
    & & \\
    \bullet \arrow[uur, dash, "\infty"] \arrow[rr, dash, "\infty"] & & \bullet \\
\end{tikzcd}
\end{figure}

This Coxeter group $W$ is crystallographic (as it contains no edges with label $4$ or $6$). It has infinitely many Cartan representations. As above, isomorphism of these representations requires $A_{ij} A_{ji}$ to be preserved for each $i, j$. So let's restrict to the case of $A_{ij} A_{ji} = 4$ for all $i, j$. 

There are now $27$ different possible Dynkin diagrams: each edge can be directed clockwise, directed anticlockwise or undirected. For ease of notation, let's write $a$ for an undirected edge, $b_+$ for a clockwise edge and $b_-$ for an anticlockwise edge. 

Again we can classify the representations according to the cyclic product $A_{12} A_{23} A_{31}$. This now takes possible values $-1,-2,-4,-8,-16,-32, -64$. As $A_{12} A_{23} A_{31} A_{21} A_{32} A_{13}$ is always $64$, the cases $-1, -64$, $-2, -32$ and $-4, -8$ correspond bijectively to each other by reversing the arrows in the diagram, or equivalently by transposing the GCM. 

In the case $A_{12} A_{23} A_{31} = -16$, also $A_{21} A_{32} A_{13} = -16$. So these GCMs are all symmetrisable. There are seven of these Dynkin diagrams: 
\pagebreak
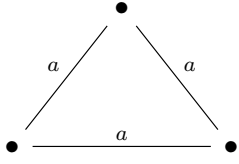
\begin{figure}[h!]
\centering
\begin{tikzcd}
    & \bullet \arrow[ddr, dash, "a"] & \\
    & & \\
    \bullet \arrow[uur, dash, "a"] \arrow[rr, dash, "a"] & & \bullet \\
\end{tikzcd}
\end{figure}

and six with one $a$, one $b_+$ and one $b_-$ such as 

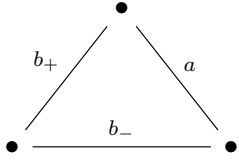
\begin{figure}[h!]
\centering
\begin{tikzcd}
    & \bullet \arrow[ddr, dash, "a"] & \\
    & & \\
    \bullet \arrow[uur, dash, "b_+"] \arrow[rr, dash, "b_-"] & & \bullet \\
\end{tikzcd}
\end{figure}

The corresponding Cartan representations are all isomorphic: we can show that the process as above of e.g. multiplying $A_{12}, A_{13}$ by $2$ and dividing $A_{21}, A_{31}$ by $2$ gives rise to an isomorphism. In this case it could transform neighbouring edges $a, a$ into $b_+, b_-$ and vice versa, or $a, b_+$ into $b_-, a$ and vice versa, or $a, b_-$ into $b_+, a$ and vice versa. 

Further, the diagram with all $a$s gives rise to the representation defined by $s_i \cdot \alpha_j = - \alpha_j$ if $i = j$ and $\alpha_j + 2 \alpha_i$ otherwise, which is isomorphic to the reflection representation of $W$. So in fact all seven of these are isomorphic to the reflection representation. 

Meanwhile, the other twenty representations have isomorphism types classified by the cyclic product, and so in particular none of them are isomorphic to the symmetrisable case or the reflection representation. 

The Dynkin diagrams are given as follows:
\begin{itemize} 
\item One each with $A_{12} A_{23} A_{31} = -1, -64$ (all $b_+$, all $b_-$ respectively).
\item Six with $A_{12} A_{23} A_{31} = -2$ (three with two $a$s and a $b_+$ and three with two $b_+$s and a $b-$).
\item Similarly by exchanging $b_+, b_-$ six with $A_{12} A_{23} A_{31} = -32$.
\item Three with $A_{12} A_{23} A_{31} = -4$ (two $b_+$s and an $a$).
\item Three with $A_{12} A_{23} A_{31} = -16$ (two $b_-$s and an $a$). 
\end{itemize}

So this means we have seven isomorphism classes of representations, corresponding to the values $-1,-2,-4,-8,-16,-32,-64$ of the cyclic product.

\chapter{The Reflection Representation}

\label{chap:4}
This chapter sets out and proves necessary and sufficient conditions for the reflection and Cartan representations of the Weyl group of a Kac-Moody algebra to be isomorphic, and also what these conditions mean from the perspective of Coxeter groups. 

We use the following notation throughout: $A$ is an $n \times n$ generalised Cartan matrix (without loss of generality indecomposable), $W$ is its Weyl group with generators $s_1, \ldots, s_n$, $\alpha_1, \ldots, \alpha_n$ are the simple roots which form a basis of the Cartan representation $H$ of $W$ (over $\mathbb{R}$) and $e_1, \ldots, e_n$ are the standard basis of the reflection representation $R$ of $W$.

The action of $W$ on $H$ is determined by $s_i \cdot \alpha_j = \alpha_j - A_{ij} \alpha_i$, and the action of $W$ on $R$ is determined by $s_i \cdot e_j = e_j + b_{ij} e_i$ where $b_{ij} = -2(e_i, e_j)$ as in Chapter \ref{chap:2}. We also have, following from the discussion in Chapter \ref{chap:2}, the following relationship between $A_{ij} A_{ji}$ and $b_{ij}$ (where $i \neq j$): 

\begin{center}
\begin{tabular}{|c|c|c|}
 \hline
 $A_{ij} A_{ji}$ & $m_{ij}$ & $b_{ij}$\\
 \hline
 $0$ & $2$ & $0$\\
 \hline
 $1$ & $3$ & $1$\\
 \hline
 $2$ & $4$ & $\sqrt{2}$\\
 \hline
 $3$ & $6$ & $\sqrt{3}$\\
 \hline
 $\geq 4$ & $\infty$ & $2$\\
 \hline
\end{tabular}
\end{center}

The approach is as follows: we'll assume that $\varphi: H \to R$ is an isomorphism, and prove several conditions which must be satisfied for this to be possible. This will give a unique possible definition of $\varphi$ (up to rescaling). Then we'll show that this in fact is well-defined, subject to the conditions proved, and thus that the conditions are also sufficient.

\begin{proposition}\label{prop:isomorphscalar} Let $\varphi: H \to R$ be an isomorphism of representations. Then for each $i$, $\varphi(\alpha_i) = \lambda_i e_i$ where $\lambda_i \in \mathbb{R} \backslash \left\{0\right\}$.\end{proposition} 

\begin{proof} We require \begin{equation*}s_i \cdot \varphi(\alpha_i) = \varphi(s_i \cdot \alpha_i) = \varphi(- \alpha_i) = - \varphi(\alpha_i)\end{equation*} So $\varphi(\alpha_i)$ belongs to the $-1$-eigenspace of the action of $s_i$. But this eigenspace is one-dimensional and spanned by $e_i$. And $\varphi(\alpha_i)$ cannot be zero as $\ker \varphi$ is trivial. \end{proof}

This result means that an isomorphism, if it exists, is determined by the choices of $\lambda_i$. Also, we can deduce relations between the $\lambda_i$:

\begin{lemma} \label{lem:isomorphcond} Suppose $i \neq j, A_{ij} \neq 0$. Then $\lambda_i = \frac{B_{ij}}{A_{ij}} \lambda_j$, where $B$ is the matrix with entries $B_{ij} = -b_{ij}$.\end{lemma}

\begin{proof} \begin{equation*}s_i \cdot \varphi(\alpha_j) = \lambda_j s_i \cdot e_j = \lambda_j e_j - \lambda_j B_{ij} e_i\end{equation*} And \begin{equation*}\varphi(s_i \cdot \alpha_j) = \varphi(\alpha_j - A_{ij} \alpha_i) = \lambda e_j - A_{ij} \lambda_i e_i\end{equation*} Comparing coefficients of $e_i$ and noting $A_{ij} \neq 0$ gives the result. \end{proof}

Since $A$ is indecomposable, this means that $\varphi$, if it exists, is uniquely determined up to rescaling. So $H$ and $R$ are isomorphic if and only if these values of $\lambda_i$ are well-defined. 

The first check is that the definitions $\lambda_i = \frac{B_{ij}}{A_{ij}} \lambda_j$ and $\lambda_j = \frac{B_{ji}}{A_{ji}} \lambda_i$ are compatible. Substituting the first definition into the second and rearranging gives the condition \begin{equation*}A_{ij} A_{ji} = (B_{ij})^2 = (b_{ij})^2\end{equation*} by symmetry of $B$. 

We can use the table above to check this for each possible value of $A_{ij} A_{ji}$ and see that it is satisfied precisely when $A_{ij} A_{ij} \in \left\{1,2,3,4\right\}$. 

\begin{theorem} Suppose $\varphi: H \to R$ is an isomorphism. Then $A$ is symmetrisable. \end{theorem}

\begin{proof} We use the alternate characterisation of symmetrisability set out in Lemma \ref{lem:symmcond}. As remarked in Chapter \ref{chap:2}, it is sufficient to consider the case in which $i_1, \ldots, i_k, i_1$ forms a circuit in the Dynkin diagram, i.e. $A_{i_j i_{j+1}} \neq 0$ for all $j$.

Let $\varphi(\alpha_i) = \lambda_i$ for each $i$. Then we can calculate 

\begin{equation*}\lambda_{i_1} = \dfrac{B_{i_1i_2}}{A_{i_1i_2}} \lambda_{i_2} = \ldots = \dfrac{ B_{i_1i_2} \ldots B_{i_k i_1}}{A_{i_1i_2} \ldots A_{i_k i_1}} \lambda_{i_1}\end{equation*}

But by assumption, $\lambda_{i_1}$ is well-defined and so

\begin{equation*}
    B_{i_1i_2} \ldots B_{i_k i_1} = A_{i_1i_2} \ldots A_{i_k i_1}
\end{equation*}

Note that for each $i, j$, $B_{ij}^2 = A_{ij} A_{ji}$ by the above condition. So if we square both sides of the equation we obtain 

\begin{equation*}A_{i_1i_2} \ldots, A_{i_k i_1} A_{i_2 i_1} \ldots A_{i_1 i_k} = (A_{i_1i_2} \ldots A_{i_k i_1})^2\end{equation*}

Since $A_{i_1i_2} \ldots A_{i_k i_1}$ is non-zero, we can divide by it to obtain

\begin{equation*}  A_{i_2 i_1} \ldots A_{i_1 i_k}= A_{i_1i_2} \ldots A_{i_k i_1}\end{equation*}

So, since the choice of circuit was arbitrary, $A$ is symmetrisable. \end{proof}

\begin{corollary} \label{cor:symmreflectcond} Suppose $A_{ij} A_{ji} \in \left\{0,1,2,3,4\right\}$ for all $i, j$. Then $A$ is symmetrisable if and only if \begin{equation*}
    B_{i_1i_2} \ldots B_{i_k i_1} = A_{i_1i_2} \ldots A_{i_k i_1}
\end{equation*} for all circuits $i_1, \ldots, i_k, i_1$ in the Dynkin diagram. \end{corollary}

\begin{theorem} \label{thm:reflectisomorph} Suppose $A_{ij} A_{ji} \in \left\{0,1,2,3,4\right\}$ for all $i, j$, and $A$ is symmetrisable. Then $H$ and $R$ are isomorphic.\end{theorem}

\begin{proof} By Lemma \ref{lem:isomorphcond}, it is sufficient to show that we can make a well-defined choice of $\lambda_1, \ldots, \lambda_n \in \mathbb{R}$ such that $\lambda_i = \frac{B_{ij}}{A_{ij}} \lambda_j$. 

Fix an arbitrary $i$ and let $\lambda_i = 1$. For each $j$, choose a path $i = k_1, k_2, \ldots, k_r = j$  from $i$ to $j$ and define values of $\lambda_{k_l}$ along this path by \begin{equation*}\lambda_{k_{l+1}} = \dfrac{B_{k_l k_{l+1}}}{A_{k_{l+1} k_l}} \lambda_{k_l}\end{equation*} This will eventually define $\lambda_j$.

Ambiguity in this definition arises where we have two distinct paths $i = k_1, \ldots, k_r = j$ and $i = l_1, \ldots, l_s = j$ from $i$ to $j$. Then we could define either \begin{equation*} \lambda_j = \dfrac{ B_{k_1k_2} \ldots B_{k_r k_1}}{A_{k_1k_2} \ldots A_{k_r k_1}} \lambda_i\end{equation*} or \begin{equation*} \lambda_j = \dfrac{B_{l_1l_2} \ldots B_{l_s l_1}}{A_{l_1l_2} \ldots A_{l_s l_1}} \lambda_i\end{equation*}

For these definitions to be consistent with each other, we need  \begin{equation*}\dfrac{B_{k_1k_2} \ldots B_{k_r k_1}}{A_{k_1k_2} \ldots A_{k_r k_1}}= \frac{B_{l_1l_2} \ldots B_{l_s l_1}}{A_{l_1l_2} \ldots A_{l_s l_1}}\end{equation*}

But if we apply Corollary \ref{cor:symmreflectcond} to the "circuit" $k_1, \ldots, k_r, k_{r-1}, \ldots, k_1$ we see that \begin{equation*}(B_{k_1k_2} \ldots B_{k_r k_1})^2 = A_{k_1k_2} \ldots A_{k_{r-1} k_r} A_{k_2k_1} \ldots A_{k_r k_{r-1}}\end{equation*} and similarly \begin{equation*}(B_{l_1l_2} \ldots B_{l_{s-1} l_s})^2 = A_{l_1l_2} \ldots A_{l_{s-1} l_s} A_{l_2l_1} \ldots A_{l_s l_{s-1}}\end{equation*}

So, squaring both sides and substituting these results, it remains to show that:

\begin{equation*} \dfrac{A_{k_2k_1}\ldots A_{k_rk_{r-1}}}{A_{k_1 k_2} \ldots A_{k_{r-1} k_r}}= \dfrac{A_{l_2l_1} \ldots A_{l_s l_{s-1}}}{A_{l_1 l_2} \ldots A_{l_{s-1} l_s}}\end{equation*}

and multiplying out denominators gives the condition for symmetrisability for the cycle $i = k_1, k_2 \ldots, k_r = j = l_s, l_{s-1}, \ldots, l_2, l_1 = i$. 

So $\lambda_1, \ldots, \lambda_n$ are well-defined. Now let $j, k$ be such that $A_{jk} \neq 0$. Choose a path $l_1, \ldots, l_r$ from $i$ to $j$. Then appending $k$ to this gives a path from $i$ to $k$. So we can use these paths to define $\lambda_j$ and $\lambda_k$, and then substituting in the resulting expression gives \begin{equation*}\lambda_k = \frac{B_{jk}}{A_{kj}} \lambda_j\end{equation*} which is the required condition.

Thus $\varphi: H \to R$ defined by $\varphi(\alpha_i) = \lambda_i e_i$ for each $i$ is the required isomorphism. \end{proof}

As a corollary of this, we obtain an alternative proof of one direction of Theorem \ref{thm:crystalcond}:

\begin{proof}[Alternative Proof]Suppose that all edge labels of $W$ are in $\left\{2,3,4,6, \infty\right\}$ and for every circuit in the Coxeter graph the numbers of edges labelled $4$ and $6$ in that circuit are even. 

Define a Dynkin diagram to have vertices $1, \ldots, n$ and edges as follows: a single edge from $i$ to $j$ wherever there is an edge labelled $3$ in the Coxeter graph, an undirected quadruple edge from $i$ to $j$ wherever there is an edge labelled $\infty$ in the Coxeter graph, and double and triple edges from $i$ to $j$ wherever there is an edge labelled $4$ or $6$ respectively in the Coxeter graph. 

By the given condition on edges labelled $4$ and $6$ in each circuit, we can choose directions of the double and triple edges such that in each circuit, there are the same number of double edges pointing clockwise as anticlockwise and the same number of triple edges pointing clockwise as anticlockwise. 

The corresponding Kac-Moody algebra has Weyl group $W$. The choices we made in assigning edge labels and directions mean that if $A$ is the generalised Cartan matrix then $A$ is symmetrisable and $A_{ij} A_{ji} \in \left\{0,1,2,3,4\right\}$ for all $i, j$. So by Theorem \ref{thm:reflectisomorph}, the Cartan and reflection representations are isomorphic. Thus $W$ is crystallographic. \end{proof}

Note also that the lattice constructed in the proof of this result by Humphreys is in fact the image of the isomorphism $\varphi: H \to R$ constructed in the proof of Theorem \ref{thm:reflectisomorph}.

\begin{corollary}Let $W$ be a Coxeter group with edge labels in $\left\{2,3,4,6,\infty\right\}$. Then $W$ is crystallographic if and only if the reflection representation of $W$ is isomorphic to some Cartan representation. \end{corollary}

Also, the condition in Corollary \ref{cor:symmreflectcond} is exactly that $A$ and $B$ are equivalent (if we generalise the definition of equivalence to matrices which may have irrational entries). 

\chapter{Isomorphism Types of Cartan Representations}

\label{chap:5}
As demonstrated in Chapter \ref{chap:3}, a given Coxeter group may have many different Cartan representations. However, two distinct Cartan representations can be isomorphic.

This chapter aims to give a condition for two Cartan representations over a fixed field $k$ of characteristic $0$ to be isomorphic. We'll follow the same general strategy as in the previous chapter. In what follows $A$ and $B$ are two indecomposable generalised Cartan matrices giving rise to Cartan representations $H_A$, $H_B$ over $k$ with bases of simple roots $\alpha_1, \ldots, \alpha_n$ and $\beta_1, \ldots, \beta_n$ respectively.

In order to even ask whether these representations are isomorphic, we first need the two Weyl groups to be the same. That is, for $k \in \left\{0,1,2,3\right\}$, $A_{ij} A_{ji} = k$ if and only if $B_{ij} B_{ji} = k$. We don't demand the same condition for $A_{ij} A_{ji} \geq 4$, but we shall see that this is necessary for $H_A$ and $H_B$ to be isomorphic. Let $s_1, \ldots, s_n$ be the generators of the Weyl group.

\begin{proposition} Suppose $\varphi: H_A \to H_B$ is an isomorphism. Then for each $i$, $\varphi(\alpha_i) = \lambda_i \beta_i$ for some scalar $\lambda_i$.\end{proposition}

\begin{proof} The proof is the same as in Proposition \ref{prop:isomorphscalar}.\end{proof}

\begin{lemma} \label{lem:scalardef} Suppose $i \neq j$, $A_{ij} \neq 0$. Then \begin{equation*}\lambda_i = \frac{B_{ij}}{A_{ij}} \lambda_j\end{equation*}\end{lemma}

\begin{proof} We require \begin{equation*}s_i \cdot \varphi(\alpha_j) = \varphi(s_i \cdot \alpha_j)\end{equation*} The left-hand side of this is given by \begin{equation*}s_i \cdot (\lambda_j \beta_j) = \lambda_j \beta_j - B_{ij} \lambda_j \beta_i\end{equation*}

And the right-hand side is \begin{equation*}\varphi(\alpha_j - A_{ij} \alpha_i) = \lambda_j \beta_j - A_{ij} \lambda_i \beta_i\end{equation*} So equating coefficients of $\beta_i$, $\lambda_i A_{ij} = \lambda_j B_{ij}$ from which the result follows. \end{proof}

This is necessary and sufficient for $\varphi$ to be an isomorphism of representations. So it suffices to find a condition for the $\lambda_i$ to be well-defined. 

\begin{theorem} \label{thm:isomorphcond} $H_A$ and $H_B$ are isomorphic if and only if $A$ and $B$ are equivalent, i.e. \begin{equation*}A_{i_1i_2} \ldots A_{i_k i_1} = B_{i_1 i_2} \ldots B_{i_k i_1}  \text{ for all } i_1, \ldots, i_k \in \left\{1, \ldots, n\right\}\end{equation*} \end{theorem}

Note that, as with the condition for symmetrisability we earlier saw, it is sufficient to consider the case when $i_1, \ldots, i_k$ forms a circuit in the Dynkin diagram (for either $A$ or $B$, the two being equivalent). Also, in particular if $k = 2$ then this means $A_{ij} A_{ji} = B_{ij} B_{ji}$ for all $i, j$.

\begin{proof} Suppose $H_A$ and $H_B$ are isomorphic. Then we have a well-defined set of scalars $\lambda_1, \ldots, \lambda_n$ such that $\lambda_i = \frac{B_{ij}}{A_{ij}} \lambda_j$ for all $i, j$ such that $A_{ij} \neq 0$. Let $i_1, \ldots, i_k, i_1$ form a circuit in the Dynkin diagram.  

Then \begin{equation*}\lambda_{i_1} = \dfrac{B_{i_1i_2}}{A_{i_1i_2}} \lambda_{i_2} = \ldots = \dfrac{B_{i_1i_2} \ldots B_{i_k i_1}}{A_{i_1i_2} \ldots A_{i_k i_1}} \lambda_i\end{equation*} This means the coefficient of $\lambda_i$ on the right-hand side must be $1$ and so \begin{equation*}A_{i_1 i_2} \ldots A_{i_k i_1} = B_{i_1 i_2} \ldots B_{i_k i_1}\end{equation*}

Suppose conversely the given condition holds. Fix some $i$ and define $\lambda_i = 1$. For each $j$, choose a path $i = k_1, k_2, \ldots, k_r = j$ from $i$ to $j$ in the Dynkin diagram. Define values of $\lambda_{k_l}$ along this path by $\lambda_{k_{l+1}} = \frac{B_{k_{l+1} k_l}}{A_{k_{l+1} k_l}}$. This will eventually define $\lambda_j$.

Ambiguity in this definition arises where we have two distinct paths $i = k_1, \ldots, k_r = j$ and $i = l_1, \ldots, l_s = j$ from $i$ to $j$. Then we could define either 

\begin{equation*}\lambda_j = \dfrac{B_{k_2 k_1} \ldots B_{k_r k_{r-1}}}{A_{k_2k_1} \ldots A_{k_r k_{r-1}}} \lambda_i\end{equation*}

or
\begin{equation*}\lambda_j = \dfrac{B_{l_2 l_1} \ldots B_{l_r l_{r-1}}}{A_{l_2l_1} \ldots A_{l_r l_{r-1}}} \lambda_i\end{equation*}

For these definitions to be consistent with each other, we need \begin{equation*}\dfrac{B_{k_2 k_1} \ldots B_{k_r k_{r-1}}}{A_{k_2k_1} \ldots A_{k_r k_{r-1}}} = \dfrac{B_{l_2 l_1} \ldots B_{l_r l_{r-1}}}{A_{l_2l_1} \ldots A_{l_r l_{r-1}}} \end{equation*}

But we can apply the given condition to the "circuit" $k_1, \ldots, k_r, k_{r-1}, \ldots, k_1$ to obtain \begin{equation*}B_{k_2 k_1} \ldots B_{k_r k_{r-1}} B_{k_1 k_2} \ldots B_{k_{r-1} k_r} = A_{k_2 k_1} \ldots A_{k_r k_{r-1}} A_{k_1 k_2} \ldots A_{k_r k_{r-1}}\end{equation*} 

Multiply top and bottom of the left-hand side by $B_{k_1k_2} \ldots B_{k_{r-1}k_r}$ and substitute the above results. We then want to show 

\begin{equation*} \frac{A_{k_1k_2} \ldots A_{k_{r-1} k_r}}{B_{k_1k_2} \ldots B_{k_{r-1} k_r}} = \frac{B_{l_2l_1} \ldots B_{l_k l_{k-1}}}{A_{l_2l_1} \ldots A_{l_s l_{s-1}}}\end{equation*} 

But multiplying out denominators gives exactly the given condition for the circuit $i = k_1, \ldots, k_r = j = l_s, l_{s-1}, \ldots, l_1 = i$. So $\lambda_j$ is well-defined for each $j$.

Now let $j, k$ be such that $A_{jk} \neq 0$. Choose a path $l_1, \ldots, l_r$ from $i$ to $j$. Then appending $k$ to this gives a path from $i$ to $k$. So we can use these paths to define $\lambda_j$ and $\lambda_k$ and then substituting in the resulting expression gives $\lambda_k = \frac{B_{kj}}{A_{kj}} \lambda_j$ as required. \end{proof}

Note that one direction of this proof is almost identical to the proof of Theorem \ref{thm:reflectisomorph}. The only difference is that the $B$ given by $B_{ij} =-  2 (e_i, e_j)$ in that theorem does not necessarily form a generalised Cartan matrix, because $B_{ij}$ can take the values $\sqrt{2}, \sqrt{3}$. Also we observe the following: 

\begin{corollary} \label{cor:dualisomorphcond} The Cartan representation from the Kac-Moody algebra associated to $A$ is isomorphic to the Cartan representation from the dual Kac-Moody algebra if and only if $A$ is symmetrisable. \end{corollary}

We can also use this result to create diagrams which represent isomorphism classes of Cartan representations. A Coxeter graph uniquely specifies the Weyl group, which may have many distinct representations. A Dynkin diagram corresponds to a specific Kac-Moody algebra and thus to a specific Cartan representation, but as we've seen distinct Dynkin diagrams can give rise to isomorphic Cartan representations.

\begin{definition} Let $A$ be an $n \times n$ generalised Cartan matrix. Then the associated Cartan representation diagram has vertex set $\left\{1, \ldots n\right\}$. There is an edge from $i$ to $j$ if and only if $A_{ij} A_{ji} \neq 0$. If this is the case, the edge is labelled with the value of $A_{ij} A_{ji}$ (if the label is omitted this is taken to mean $A_{ij} A_{ji} = 1$.) Edges are not directed.

Additionally, each cycle $i_1, \ldots,i_k, i_1$ with $i_1,\ldots, i_k, i_1$ in clockwise order is labelled with the product $A_{i_1i_2} \ldots A_{i_ki_1}$.
\end{definition}

By Theorem \ref{thm:isomorphcond} this diagram uniquely determines the isomorphism class of the Cartan representation given by $A$, since each cyclic product $A_{i_1i_2} \ldots A_{i_ki_1}$ can be broken down as a product of cyclic products corresponding to cycles and of "short cyclic products" $A_{ij} A_{ji}$. 

Note also that we can recover the Coxeter graph of the Weyl group from this diagram by deleting all cycle labels, and replacing edge labels $2, 3$ with $4, 6$ respectively and edge labels $\geq 4$ with $\infty$. 

For example, if we take the Coxeter graph $\Gamma$ from Chapter \ref{chap:3}:

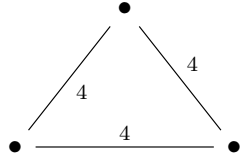
 \begin{figure}[h]
\centering
\begin{tikzcd}
& \bullet \arrow[ddr, dash, "4"] \arrow[ddl, dash, "4"] & \\
& & \\
\bullet \arrow[rr, dash, "4"] & & \bullet\\
\end{tikzcd}
\end{figure}

The possible corresponding Cartan representation diagrams are as follows (where the bold label in the centre of the triangle is the cyclic product $A_{12} A_{23} A_{31}$):

 
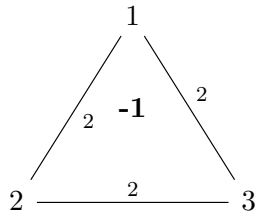
\begin{figure}[h]
\centering
\begin{tikzcd}
& 1 \arrow[ddr, dash, "2"] \arrow[ddl, dash, "2"] & \\
& \textbf{-1} & \\
2 \arrow[rr, dash, "2"] & & 3\\
\end{tikzcd}
\end{figure}

 
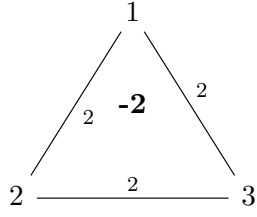
\begin{figure}[h]
\centering
\begin{tikzcd}
& 1 \arrow[ddr, dash, "2"] \arrow[ddl, dash, "2"] & \\
& \textbf{-2} & \\
2 \arrow[rr, dash, "2"] & & 3\\
\end{tikzcd}
\end{figure}

 
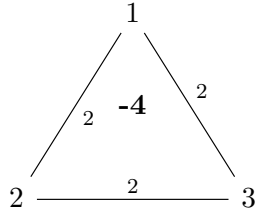
\begin{figure}[h]
\centering
\begin{tikzcd}
& 1 \arrow[ddr, dash, "2"] \arrow[ddl, dash, "2"] & \\
& \textbf{-4} & \\
2 \arrow[rr, dash, "2"] & & 3\\
\end{tikzcd}
\end{figure}

 
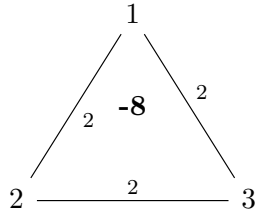
\begin{figure}[h!]
\centering
\begin{tikzcd}
& 1 \arrow[ddr, dash, "2"] \arrow[ddl, dash, "2"] & \\
& \textbf{-8} & \\
2 \arrow[rr, dash, "2"] & & 3\\
\end{tikzcd}
\end{figure}
\pagebreak
Also, recall that in Chapter \ref{chap:3} we had a notion of transforming between Dynkin diagrams to give isomorphic representations. We can now formalise this in terms of equivalence of generalised Cartan matrices. Note firstly that if $D$ is an $n \times n$ diagonal matrix with positive entries, we can write it as a product $D_1 \ldots. D_n$ where the $D_i$ are diagonal and $(D_i)_{jj} = 1$ if $i \neq j$. This means that we can transform between any two equivalent GCMs (and thus isomorphic representations) using only these $D_i$. 

Fix $i$, and let $\lambda$ be the $(i, i)$th entry of $D_i$. Let $A$ be a generalised Cartan matrix and $B = D_i A D_i^{-1}$. Then $B_{jk} = A_{jk}$ if $j, k \neq i$, $B_{ji} = \lambda^{-1} A_{ji}$ for $j \neq i$ and $B_{ik} = \lambda A_{ik}$ for $k \neq i$. (Also $B_{ii} = \lambda^{-1} \lambda A_{ii} = A_{ii} = 2$.) 

That is, for all edges $(i, k)$ leaving the vertex $i$ we multiply $A_{ik}$ by $\lambda$ and divide $A_{ki}$ by $\lambda$. But for $B$ to be a generalised Cartan matrix, we need the entries of $B$ to remain integers. So (assuming $\lambda \neq 1$ so $D_i$ is not the identity matrix) either $\lambda > 1$ and $\lambda$ divides every $A_{ji}, j \neq i$, or $\lambda < 1$ and $\frac{1}{\lambda}$ divides every $A_{ik}, k \neq i$. 

For instance, in the example given above and in Chapter \ref{chap:3} we could transform from the first of the below diagrams to the second by "flipping" the directions of the two edges. This corresponds to taking $\lambda = 2$. 

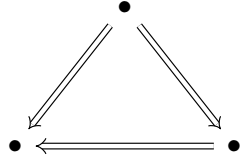
\begin{figure}[h]
\centering
\begin{tikzcd}
    & \bullet \arrow[ddr, Rightarrow] \arrow[ddl, Rightarrow] & \\
    & & \\
    \bullet  & & \bullet \arrow[ll, Rightarrow]  \\
\end{tikzcd}
\end{figure}

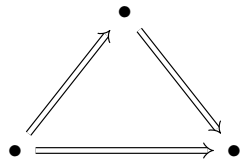
\begin{figure}[h]
\centering
\begin{tikzcd}
    & \bullet \arrow[ddr, Rightarrow] & \\
    & & \\
    \bullet \arrow[uur,Rightarrow] \arrow[rr, Rightarrow] & & \bullet \\
\end{tikzcd}
\end{figure}

\chapter{Invariant Bilinear Forms and the Dual Representation}

\label{chap:6}
We now consider the dual of a Cartan subalgebra representation. We'll begin with proving some facts about dual representations in general.

\begin{definition}Let $V$ be a representation of a group $G$. Then the dual representation $V^*$ is given by the vector dual of $V$. If $\theta \in V^*$ then $g \in G$ acts on $\theta$ by $g \cdot \theta (v) = \theta(g^{-1} \cdot v)$. If we fix a basis $e_1, \ldots, e_n$ of $V$ and $g$ acts as the matrix $A$ with respect to this basis, then $g$ acts as $(A^{-1})^T$ with respect to the dual basis $e_1^*, \ldots, e_n^*$. \end{definition}

\begin{proposition}\label{prop:isomorphiffform} Let $V$ be a representation of a group $G$ and $V^*$ its dual. Then $V$ and $V^*$ are isomorphic if and only if there exists a non-degenerate $G$-invariant bilinear form $(\cdot, \cdot)$ on $V$. \end{proposition}

The proof is omitted, but we note here that given a form $(\cdot, \cdot)$ the corresponding isomorphism maps $v \to (v, \cdot)$ and given an isomorphism $\varphi$ the corresponding form is $(v, w) = \varphi(v)(w)$. Observe also that there is no requirement for the form to be symmetric. 

Now let's consider the particular case of the Cartan representation $H$ of a Weyl group $W$. By the above proposition, $H$ and $H^*$ are isomorphic if and only if there exists a non-degenerate $W$-invariant bilinear form on $H$. So we aim to establish conditions for the existence of such a form.

We use the same notation as in previous chapters, and let $(\cdot, \cdot)$ be a bilinear form on $H$. 

$(\cdot, \cdot)$ is $W$-invariant if and only if \begin{equation*}(s_i \cdot \alpha_j, s_i \cdot \alpha_k) = (\alpha_j, \alpha_k) \text{ for all } i, j, k\end{equation*} if and only if \begin{equation*}(\alpha_j - A_{ij} \alpha_i, \alpha_k - A_{ik} \alpha_k) = (\alpha_j, \alpha_k) \text{ for all } i, j, k\end{equation*} if and only if \begin{equation*}-(\alpha_i, \alpha_k) A_{ij} - (\alpha_j, \alpha_i) A_{ik} + (\alpha_i, \alpha_i) A_{ij} A_{ik} = 0 \text{ for all } i,j,k\end{equation*}

Let $B$ be the matrix of the form in this basis, so $B_{ij} = (\alpha_i, \alpha_j)$. Then we rewrite the above as \begin{equation}\label{eq:formcond}B_{ik} A_{ij} + B_{ji} A_{ik} = A_{ij} A_{ik} B_{ii} \text{ for all } i,j,k.\end{equation}

Suppose now $A$ is a Cartan matrix, corresponding to the case where the Kac-Moody algebra $\mathfrak{g}(A)$ is in fact a finite-dimensional semisimple Lie algebra. Here we have a non-degenerate symmetric $W$-invariant bilinear form induced by the Killing form and \begin{equation} \label{eq:finiteform} A_{ij} = 2 \frac{B_{ij}}{B_{ii}}\end{equation} (\cite{Carter}, p.71)

We observe that this definition satisfies the relation \eqref{eq:formcond} between $A$ and $B$, but that this relies on symmetry of $B$ where we have not assumed symmetry of the form in the general case. Also, the given definition of $A$ in terms of $B$ cannot be easily reversed to give a definition of $B$ in terms of $A$, especially when $B_{ii}$ is not necessarily non-zero.

So instead, we'll look at simpler cases of the expression \eqref{eq:formcond}. First suppose $i = j$ (so that $A_{ij} = 2$) to obtain

\begin{equation*}2 B_{ik} + B_{ii} A_{ik} = 2 A_{ik} B_{ii}\end{equation*}

and thus, cancelling, \begin{equation} \label{eq:simpleformcond1} A_{ik} B_{ii} = 2 B_{ik}\end{equation} 

which is in fact a rearrangement of the condition \eqref{eq:finiteform} in the finite case. Similarly by taking $i = k$, \begin{equation} \label{eq:simpleformcond2} A_{ij} B_{ii} = 2 B_{ji}\end{equation}

These conditions together imply that $B_{ij} = B_{ji}$ for all $i, j$. That is:

\begin{proposition}\label{prop:symmform} Let $(\cdot, \cdot)$ be a $W$-invariant bilinear form on $H$. Then it is symmetric.\end{proposition}

Also, we can check by direct calculation that \eqref{eq:simpleformcond1} and \eqref{eq:simpleformcond2}together imply \eqref{eq:formcond}. So we can reduce to a single condition, reminiscent of the finite case:

\begin{theorem} \label{thm:invariantformcond} Let $B$ be the matrix of a symmetric bilinear form on $H$ in the basis $\left\{\alpha_1, \ldots, \alpha_n\right\}$. Then the form is $W$-invariant if and only if $A_{ij} B_{ii} = 2 B_{ij}$ for all $i, j$.\end{theorem}

This means that we can uniquely (up to rescaling) define $B$ as follows: fix $i$ and set $B_{ii} = 1$. Given $j$, choose a path $i = k_1, \ldots, k_r = j$ from $i$ to $j$ in the Dynkin diagram. Set \begin{equation*}B_{k_1k_2} = B_{k_2k_1} = \dfrac{1}{2} A_{k_1k_2} B_{k_1k_1}\end{equation*} Then set \begin{equation*}B_{k_2k_2} = \dfrac{2}{A_{k_2k_1}} B_{k_1k_2}\end{equation*} Repeat this process along the path until we eventually define $B_{jj}$ and $B_{jl}$ for any $l$ neighbouring $j$. 

\begin{lemma} Suppose a form with matrix $B$ as above is well-defined. Then $A$ is symmetrisable. \end{lemma} 
\begin{proof} Note that if $A_{ij} \neq 0$, \begin{equation*}B_{jj} = \dfrac{2}{A_{ji}} B_{ij} = \dfrac{A_{ij}}{A_{ji}} B_{ii}\end{equation*} Let $i_1, \ldots, i_k, i_1$ be a circuit in the Dynkin diagram. Then applying the above result iteratively we see that \begin{equation*}B_{i_1 i_1} = \frac{A_{i_1i_2} \ldots A_{i_ki_1}}{A_{i_2i_1} \ldots A_{i_1i_k}} B_{i_1 i_1}\end{equation*} 

 So the condition for this to be well-defined is exactly \begin{equation*}A_{i_1i_2} \ldots A_{i_ki_1} = A_{i_2i_1} \ldots A_{i_1i_k}\end{equation*} the symmetrisability condition for this circuit. \end{proof}

 Now we have a form in the case where $A$ is symmetrisable. We'll here summarise the known theory of Kac-Moody algebras that leads to this result. 
 
\begin{definition} (\cite{Carter}, pp360-361) Let $A$ be a symmetrisable generalised Cartan matrix and write $A = DB$ where $D = \text{diag}(d_1, \ldots, d_n)$ is diagonal and non-singular and $B$ is symmetric. Let $(\mathfrak{h}, \Pi, \Pi^\wedge)$ be a minimal realisation of $A$, with $\Pi = \left\{\alpha_1, ... \alpha_n\right\}$ and $\Pi^\wedge = \left\{h_1, ... h_n\right\}$. Extend $h_1, \ldots, h_n$ to a basis $h_1, \ldots, h_n, x_1, \ldots, x_{n-l}$ of $\mathfrak{h}$. Define a symmetric bilinear form $\langle \cdot, \cdot \rangle: H \times H \to \mathbb{C}$ by 
\begin{itemize}
    \item $\langle h_i, h_j \rangle = d_i d_j B_{ij}$ for $i, j = 1, \ldots, n$,
    \item $\langle h_i, x_j \rangle = \langle x_j h_i \rangle = d_i \alpha_i(x_j)$ for $i = 1, \ldots, n, j = 1, \ldots, n -l$, and,
    \item $\langle x_i, x_j \rangle = 0$ for $i, j = 1, \ldots, n -l$. \end{itemize} \end{definition}

 \begin{proposition} (Proposition 16.1 of \cite{Carter}, p.361) This form on $\mathfrak{h}$ is non-degenerate. \end{proposition}

 Proof is omitted, but see \cite{Carter}, pp.361-363.

 This gives rise to an isomorphism $\varphi: \mathfrak{h} \to \mathfrak{h}^*$ given by $\varphi(h) = t_h = \langle h, \cdot \rangle$. Thus it induces also a symmetric bilinear form $(\cdot, \cdot)$ on $\mathfrak{h}^*$ defined by $(t_{h_1}, t_{h_2}) = \langle h_1, h_2 \rangle$.

\begin{proposition} (\cite{Carter}, pp.373-374) $(\cdot, \cdot)$ is $W$-invariant. \end{proposition}

 Proof is once again omitted but may be found in \cite{Carter}, pp.373-374. 

 Now $(\cdot, \cdot)$ restricts to a symmetric bilinear $W$-invariant form on the Cartan representation $H$ of $W$. This, by the above, is unique up to rescaling. So $H$ and its dual are isomorphic if and only if this $(\cdot, \cdot)$ is non-degenerate. 

\begin{lemma} $\varphi(h_i) = d_i \alpha_i$. \end{lemma}

\begin{proof}By construction of $\varphi$ and the form on $\mathfrak{h}$, they agree on the basis $h_1, \ldots, h_n, x_1, \ldots, x_{n-l}$ of $H$.\end{proof}

So the space spanned by $\varphi(h_1), \ldots, \varphi(h_n)$ is $H$. With respect to this basis of $H$, the matrix of $(\cdot, \cdot)$ is $B$, which is such that $A = DB$ where $D$ is diagonal and non-singular. Thus the form is non-degenerate if and only if $\det B = 0$ if and only if $\det A = 0$.

\begin{theorem} \label{thm:dualisomorphcond} Let $A$ be a generalised Cartan matrix. Then $H_A$ and its dual are isomorphic if and only if $\det A \neq 0$ and $A$ is symmetrisable.\end{theorem}

There is no known condition in terms of the Dynkin diagram for the determinant of a generalised Cartan matrix to be zero, though Theorem \ref{thm:gcmclassification} gives a partial result: finite-type GCMs never have determinant $0$ and affine ones always do. 

We're also interested in whether the dual representation can be isomorphic to some Cartan representation, even if not the original representation. Let $a_1, \ldots, a_n$ be the dual basis of $\alpha_1, \ldots, \alpha_n$, and let $S_i$ be the matrix of $s_i$ acting on the dual space, with respect to this basis. 

Then $S_i^T$ is the matrix of $s_i$ acting on $H$ with respect to the basis $\alpha_1, \ldots, \alpha_n$. We can write down the entries of this matrix: \begin{equation*}(S_i^T)_{jk} = \begin{cases}
    -1, & \text{if }  j = k = i\\
    1, & \text{if }j = k \neq i\\
    -A_{ik}, & \text{if } j = i \neq k\\
    0, & \text{otherwise}
\end{cases}\end{equation*}

This means that \begin{equation*}(S_i)_{jk} = \begin{cases}
    -1, & \text{if } j = k = i\\
    1, & \text{if } j = k \neq i\\
    -A_{ij}, & \text{if }k = i \neq j\\
    0, & \text{otherwise}
\end{cases}\end{equation*}

i.e. the identity except with the $i$th column replaced by $-1$ in the $i$th entry and $-A_{ij}$ in the $j$th entry for $j \neq i$.

Now this is isomorphic to some Cartan matrix if and only if we can transform to a basis $\left\{e_1, \ldots, e_n\right\}$ such that the $s_i$ act on the $e_i$ by $s_i \cdot e_j = e_j - B_{ij} e_i$ for some generalised Cartan matrix $B$ with Weyl group $W$. In particular, we require each $e_i$ to be an eigenvector of $S_i$ with eigenvector $-1$.

Now $S_i$ has determinant $-1$ since $S_i^T$ does, and has a $1$-eigenspace of dimension $n-1$, so it must have a one-dimensional $-1$-eigenspace. We can in fact solve the system of linear equations to find the eigenvector: suppose that

\begin{equation*}S_i \begin{pmatrix}\lambda_1 \\ \ldots\\ \lambda_n\\\end{pmatrix} = -\begin{pmatrix}\lambda_1 \\ \ldots\\ \lambda_n\\\end{pmatrix}  \end{equation*}

Then we obtain the equations $- \lambda_i = - \lambda_i$ and, for each $j \neq i$, \begin{equation*}\lambda_j = \dfrac{1}{2} \lambda_i A_{ij}\end{equation*} It is convenient to scale by taking $\lambda_i = 2$ so that $\lambda_j = A_{ij}$ for each $j$. That is,

\begin{equation*}e_i = \Sigma_{j=1}^n A_{ij} a_j = A a_i\end{equation*}

So in particular, if the columns of $A$ are linearly dependent then so are the $e_i$, meaning that the $e_i$ cannot possibly form a basis of $H^*$.

\begin{theorem}\label{thm:det0notisomorph} Suppose $\det A = 0$. Then the dual representation of $H_A$ is not isomorphic to any Cartan representation. \end{theorem}

We'll thus suppose further that $\det A \neq 0$. Then we can transform to a basis $\left\{e_1, \ldots, e_n\right\}$ where $e_i$ is a $-1$-eigenvector of $S_i$ for each $i$. Now we want to consider $S_i \cdot e_j$. We'll calculate this in the dual basis and transform afterwards.

\begin{equation*}S_i \cdot e_j = \Sigma_{k=1}^n (A_{jk} a_k - A_{ji} A_{ik} a_k) = e_j - A_{ji} e_k\end{equation*} 

This is a Cartan representation as above, with generalised Cartan matrix $A^T$ (which has Weyl group $W$.) 

\begin{theorem}\label{thm:dualtransposeisomorph} Suppose $A$ is a generalised Cartan matrix with Weyl group $W$ and $\det A \neq 0$. Then the dual representation of $A$ is isomorphic to the Cartan representation of $A^T$.\end{theorem}

Note that this and \ref{cor:dualisomorphcond} give an alternate proof of \ref{thm:dualisomorphcond}.

\chapter{Fields of Finite Characteristic}

\label{chap:7}
In the previous chapters, we restricted to the case where the Cartan representation is over a field of characteristic $0$. But the results therein do largely hold over fields of finite characteristic, with some additional considerations which we set out in this chapter. The main exception is the reflection representation of a Coxeter group, which is naturally defined over $\mathbb{R}$. Its definition can be extended to other fields, but doing so is beyond the scope of this text. 

In what follows $A$ is a generalised Cartan matrix with Weyl group $W$ generated by the reflections $s_1, \ldots, s_n$, and $k$ is a field of finite characteristic $p$. The simple roots $\alpha_1,\ldots, \alpha_n$ and the Cartan representation $H$ of $W$ over $k$ are defined as in Chapter \ref{chap:2} and as used throughout.

Note that the representation only depends on the values of $A_{ij}$ modulo $p$. So we want to define a notion of reducing a GCM modulo $p$. We can't always do this, however.

\begin{definition}Let $A$ be a generalised Cartan matrix and $p$ a prime. Then we say that $A$ is reducible modulo $p$ if whenever $p$ divides $A_{ij}$, $p$ also divides $A_{ji}$.\end{definition}

We can think of this as the analogue of $A_{ij} = 0$ if and only if $A_{ji} = 0$. 

\begin{definition}Let $A$ be a generalised Cartan matrix which is reducible modulo $p$ for some prime $p$. Then the reduction of $A$ mod $p$ is the matrix $\overline{A}$ where $\overline{A}_{ij} = A_{ij} + p \mathbb{Z}$. \end{definition}

Because we insist on it being reducible, we can then think of this as a generalised Cartan matrix with entries in $\mathbb{F}_p$. Evidently two GCMs with the same reduction modulo $p$ will have isomorphic Cartan representations over a field of characteristic $p$.

We can also define the Dynkin diagram modulo $p$ of a GCM $A$ to be the Dynkin diagram of the unique GCM $A'$ with $\overline{A} = \overline{A'}$ and $A'_{ij} \in \left\{0, -1, \ldots,-(p-1)\right\}$ for $i \neq j$.

Now for this class of GCM, the results of Chapter \ref{chap:5} generalise straightforwardly as follows. Fix a prime $p$ and a field $k$ of characteristic $p$. Let $A, B$ be GCMs which are reducible modulo $p$ and indecomposable, with Cartan representations $H_A, H_B$ respectively over $k$, and the same Weyl group $W$.

\begin{proposition} \label{prop:isomorphscalarmodp} Suppose $\varphi: H_A \to H_B$ is an isomorphism. Then for each $i$, $\varphi(\alpha_i) = \lambda_i \beta_i$ for some scalar $\lambda_i$. \end{proposition}

\begin{proof}
    For $p \neq 2$ the proof is the same as that of Proposition \ref{prop:isomorphscalar}. For $p = 2$ this argument fails as the $-1$-eigenspace is also the $1$-eigenspace which is not in general one-dimensional. Instead we use the following argument:

    Fix $i$ and choose some $j$, which must exist since $A$ is reducible modulo $2$ and indecomposable, such that \begin{equation*}s_i \cdot \alpha_j \neq \alpha_j\end{equation*} 

    Then \begin{equation*} s_i \cdot \alpha_j = \alpha_i + \alpha_j\end{equation*}

    so \begin{equation*} \varphi(s_i \cdot \alpha_j) = \varphi(\alpha_i) + \varphi(\alpha_j)\end{equation*}

    Now $s_i \cdot \varphi(\alpha_j)$ must be equal to $\varphi(\alpha_j)$ plus some (possibly zero) multiple of $\beta_i$. Hence $\varphi(\alpha_i)$ is a scalar multiple of $\beta_i$ as required.     
\end{proof}

\begin{lemma} \label{lem:scalardefmodp} Suppose $i \neq j$ and $p$ does not divide $A_{ij}$. Then $\lambda_i = \frac{B_{ij}}{A_{ij}} \lambda_j$ (up to reduction of $A_{ij}, B_{ij} \mod p$). 
\end{lemma}

\begin{proof}
    The proof is the same as in Lemma \ref{lem:scalardef}.
\end{proof}

Note that we do need the stronger condition that $p$ does not divide $A_{ij}$, so that we do not divide by zero modulo $p$. Since we know that the $\lambda_i$ are non-zero, this means that if $p$ does not divide $A_{ij}$ it cannot divide $B_{ij}$ either. And by considering $\varphi^{-1}$, this means that:

\begin{corollary}
\label{cor:divisiblecond}
    Suppose $H_A$ and $H_B$ are isomorphic. Then for each $i, j$, $p$ divides $A_{ij}$ if and only if $p$ divides $B_{ij}$.
\end{corollary}

\begin{corollary}
    Suppose $H_A$ and $H_B$ are isomorphic. Then if $A$ is reducible modulo $p$, so is $B$. 
\end{corollary}

Note that this still holds in the case $p = 2$, because we assume reducibility of $A$ but not $B$ to obtain Proposition \ref{prop:isomorphscalarmodp}.

\begin{theorem}Let $A, B$ be generalised Cartan matrices with $A$ reducible modulo $p$ and $H_A, H_B$ be the respective Cartan representations over a field of characteristic $p$. Then $H_A$ and $H_B$ are isomorphic if and only if \begin{equation*}A_{i_1i_2} \ldots A_{i_{k} i_1} \equiv B_{i_1i_2} \ldots B_{i_ki_1} \mod p \text{ for all } i_1, \ldots , i_k \in \left\{1, \ldots, n\right\}.\end{equation*}  
\end{theorem}

\begin{proof}
As in Theorem \ref{thm:isomorphcond}.
\end{proof}

If we relax the assumption of reducibility$\mod p$, then this condition is still necessary for an isomorphism to exist if $p \neq 2$. But the argument that it is sufficient will fail, as the proof uses the assumption that for any $i, j$ there is a sequence $i = k_1, \ldots, k_r = j$ with $A_{k_{i+1} k_i}$ non-zero for each $i$, and also that $p$ divides $A_{ij}$ if and only if $p$ divides $A_{ji}$. 

It turns out that the second of these assumptions is not needed, and so we can generalise the full result to a larger class of GCMs. 

Suppose $A$ is a GCM which is not necessarily reducible modulo $p$. Then we can define another notion of the Dynkin diagram modulo $p$: the directed graph with vertices $\left\{1, \ldots, n\right\}$ and an edge from $i$ to $j$ labelled by $-A_{ji} \mod p$ wherever $p$ does not divide $A_{ji}$.

\begin{definition}We say that $A$ is indecomposable modulo $p$ if for every $i, j \in \left\{1, \ldots, n\right\}$ there exists a path from $i$ to $j$ in this Dynkin diagram. That is, a sequence $i = k_1, \ldots, k_r = j$ with $A_{k_{i+1} k_i}$ not divisible by $p$ for each $i$.\end{definition}

Note that unlike the characteristic zero case, it is not true that every generalised Cartan matrix can be decomposed into indecomposable GCMs (see the example below, for instance). 

\begin{theorem}Let $A, B$ be generalised Cartan matrices with $A$ indecomposable modulo $p$. Then $H_A, H_B$ are isomorphic if and only if \begin{equation*}A_{i_1i_2} \ldots A_{i_{k} i_1} \equiv B_{i_1i_2} \ldots B_{i_ki_1} \mod p \text{ for all } i_1, \ldots , i_k \in \left\{1, \ldots, n\right\} .\end{equation*}\end{theorem}

\begin{proof}
    We've showed the necessity of this condition already. Conversely, note that by \ref{cor:divisiblecond} $B$ is also indecomposable modulo $p$, and that the argument in \ref{lem:scalardefmodp} for $p = 2$ is still valid with this weaker assumption. Choose some $i$ and define $\lambda_i = 1$. Then, as before, given $j$ choose a path $i = k_1, \ldots, k_r = j$ from $i$ to $j$ and define $\lambda_{k_i}$ along this path, so that 

    \begin{equation*}\lambda_j = \dfrac{B_{k_2 k_1} \ldots B_{k_r k_{r-1}}}{A_{k_2k_1} \ldots A_{k_r k_{r-1}}} \lambda_i\end{equation*}

    To check that the $\lambda_j$ are well-defined, suppose we have another path $i = l_1, \ldots, l_s = j$ from $i$ to $j$. Then we want to show that \begin{equation*}\dfrac{B_{k_2 k_1} \ldots B_{k_r k_{r-1}}}{A_{k_2k_1} \ldots A_{k_r k_{r-1}}} \equiv \dfrac{B_{l_2 l_1} \ldots B_{l_r l_{r-1}}}{A_{l_2l_1} \ldots A_{l_r l_{r-1}}} \mod p\end{equation*}

    But we know that there exists a path $j = m_1, \ldots m_t = i$ from $i$ to $j$. Now $i = k_1, \ldots, k_r = j = m_1, \ldots, m_t = i$ forms a cycle in the Dynkin diagram, so we can apply the cyclic product condition:

    \begin{equation*}A_{k_2k_1} ... A_{k_r k_{r-1}} A_{m_2m_1} ... A_{m_t m_{t-1}} \equiv B_{k_2k_1} ... B_{k_r k_{r-1}} B_{m_2m_1} ... B_{m_t m_{t-1}} \mod p\end{equation*}

    Rearranging this gives \begin{equation*}\dfrac{B_{k_2k_1} ... B_{k_r k_{r-1}}}{A_{k_2k_1} ... A_{k_r k_{r-1}}} \equiv \dfrac{A_{m_2m_1} ... A_{m_t m_{t-1}}}{B_{m_2m_1} ... B_{m_t m_{t-1}}} \mod p \end{equation*}

    And similarly \begin{equation*}\dfrac{B_{l_2l_1} ... B_{l_s l_{s-1}}}{A_{l_2l_1} ... A_{l_s l_{s-1}}} \equiv \dfrac{A_{m_2m_1} ... A_{m_t m_{t-1}}}{B_{m_2m_1} ... B_{m_t m_{t-1}}} \mod p \end{equation*}

    which gives the required equivalence. So $H_A$ and $H_B$ are isomorphic. 
\end{proof}
This argument also works in the cases of Theorem \ref{thm:isomorphcond}. 

\begin{definition}Analogous to the characteristic $0$ case, we'll define the cyclic products modulo $p$ of an $n \times n$ GCM $A$ to be the quantities
\begin{equation*}A_{i_1i_2} \ldots A_{i_ki_1} \mod p\end{equation*} for $i_1, \ldots, i_k \in \left\{1, \ldots, n\right\}$ distinct.

We'll also say that $A, B$ are equivalent modulo $p$ if there exists a diagonal matrix $D$ with integer entries not divisible by $p$ such that $A_{ij} \equiv (DBD^{-1})_{ij} \mod p$ for all $i, j$ (note that the condition for the entries of $D$ to be positive loses its meaning modulo $p$).
\end{definition}

\begin{proposition}$A$ and $B$ are equivalent modulo $p$ if and only if they have the same cyclic products modulo $p$ and $p$ divides $A_{ij}$ precisely when $p$ divides $B_{ij}$.\end{proposition}

\begin{proof}
    As in \ref{prop:equivcond} (note that the second condition is immediate from equivalence modulo $p$, and that we need it for the converse argument to hold).
\end{proof}

However, attempts to weaken this condition further to allow matrices which are not indecomposable modulo $p$ will fail. For instance, let's suppose $p = 3$ and consider the below "directed Dynkin diagrams modulo $3$" corresponding to generalised Cartan matrices $A, B$ respectively: 

 \begin{figure}[h]
\centering
\begin{tikzcd}
& 1 \arrow[ddr, "1"] \arrow[ddl, "1"] & \\
&  & \\
2 \arrow[rr,  "1"] & & 3\\
\end{tikzcd}
\scalebox{1.25}{
$\begin{pmatrix}2 & -3 & -3\\-1 & 2 & -3\\-1 & -1 & 2\\\end{pmatrix}$
}
\end{figure}

 \begin{figure}[h]
\centering
\begin{tikzcd}
& 1 \arrow[ddr, "2"] \arrow[ddl, "1"] & \\
&  & \\
2 \arrow[rr,  "1"] & & 3\\
\end{tikzcd}
\scalebox{1.25}{
$\begin{pmatrix}2 & -3 & -3\\-1 & 2 & -3\\-2 & -1 & 2\\\end{pmatrix}$
}
\end{figure}

Neither of these contain any (directed) cycles, so their cyclic products are trivially equal. And the edges present in both are the same. However, we claim they are not isomorphic. Suppose $\varphi: H_A \to H_B$ is an isomorphism (working over the field $\mathbb{F}_3$ of characteristic $3$). Then $\varphi(\alpha_i) = \lambda_i \beta_i$ for $i = 1,2,3$. Without loss of generality $\lambda_1 = 1$.

Then $\lambda_2 = \frac{B_{21}}{A_{21}} \lambda_1 = 1$ and $\lambda_3 = \frac{B_{21}}{A_{21}} = 2$. But also, $\lambda_3 = \frac{B_{32}}{A_{32}} \lambda_2 = \lambda_2$, a contradiction. 

Let's also try to generalise the results of Chapter \ref{chap:6} to the case of finite characteristic. Proposition \ref{prop:isomorphiffform} does not depend on the characteristic of the field. But it is difficult to generalise the theory of Kac-Moody algebras that the initial approach relied on to finite characteristic. Instead, we'll generalise the alternative argument working directly with the dual representation.

\begin{definition}Let $A$ be an $n \times n$ generalised Cartan matrix and $p$ a prime. Then we say that $A$ is symmetrisable modulo $p$ if for all $i_1, \ldots, i_k \in \left\{1, \ldots, n\right\}$, \begin{equation*}A_{i_1i_2} \ldots A_{i_k i_1} \equiv A_{i_2i_1} \ldots A_{i_1 i_k} \mod p\end{equation*}\end{definition}

\begin{proposition}
    Suppose $A$ is indecomposable and symmetrisable modulo $p$. Then $A^T$ is indecomposable modulo $p$, and the Cartan representations we obtain from $A$ and $A^T$ over a field $k$ of characteristic $p$ are isomorphic. 
\end{proposition}

\begin{proof}
    The directed Dynkin diagram modulo $p$ for $A^T$ is obtained by reversing all arrows in the directed Dynkin diagram modulo $p$ for $A$, so $A$ is indecomposable modulo $p$ if and only if $A$ is. The second statement is the finite-characteristic case of \ref{cor:dualisomorphcond} and is proved in exactly the same way.
\end{proof}

\begin{theorem}Suppose $p \neq 2$. Let $A$ be a generalised Cartan matrix, $H_A$ the Cartan representation of the Weyl group $W$ over a field $k$ of characteristic $p$ and $H_A^*$ its dual. Then if $p$ divides $\det A$, $H_A^*$ is not isomorphic to any Cartan representation of $W$. If this is not the case and additionally $A$ is indecomposable modulo $p$, then $H_A^*$ is isomorphic to $H_{A^T}$.\end{theorem} 

\begin{proof}
    As in \ref{thm:det0notisomorph} and \ref{thm:dualtransposeisomorph}, except that instead of $A$ we use the matrix $\overline{A}$ with entries $A_{ij} \mod p$ which are taken to be in $k$.
\end{proof}

 Note that the argument fails in the case $p = 2$ as it relies on the $-1$-eigenspace being one-dimensional. In fact, in this case all eigenvalues of the matrix $S_i$ representing the action of $s_i$ in the dual basis are $1$. So there is no clear choice of which to use in an attempt to make a Cartan representation. We'll continue to assume $p \neq 2$ in the below. 

\begin{corollary}
    Let $A$ be a generalised Cartan matrix which is indecomposable modulo $p$. Then $H_A$ and $H_A^*$ are isomorphic if and only if $p$ does not divide $\det A$ and $A$ is symmetrisable modulo $p$.
\end{corollary}

\begin{corollary}
    Let $A$ be a generalised Cartan matrix which is indecomposable modulo $p$, with Weyl group $W$ and Cartan representation $H_A$. Then $H_A$ admits a non-degenerate $W$-invariant bilinear form if and only if $p$ does not divide $\det A$ and $A$ is symmetrisable modulo $p$. 
\end{corollary}

We can also show using the same argument as in Chapter \ref{chap:6} that if such a form exists, it is necessarily symmetric. 

\chapter{Conclusion}

\label{chap:8}
We here summarise the results proved, and suggest areas in which they can be extended. 

\begin{theorem} Let $A, B$ be indecomposable generalised Cartan matrices, and $k$ a field of characteristic $0$ or finite characteristic $p$. Suppose additionally that if $k$ has characteristic $p$, $A$ is indecomposable modulo $p$. Then the following are equivalent:
\begin{itemize}
    \item $A$ and $B$ are equivalent (if $\mathrm{char} k = 0$) or equivalent modulo $p$ (if $\mathrm{char} k = p)$
    \item All cyclic products (modulo $p$) of $A$ and $B$ are equal.
    \item The Cartan subalgebra representations determined by $A$ and $B$ over $k$ are (of the same Weyl group and) isomorphic. 
\end{itemize}
\end{theorem}
\begin{theorem} Let $A$ be an indecomposable (or indecomposable modulo $p$) generalised Cartan matrix and $W$ the associated Weyl group. Then the following are equivalent:

\begin{itemize}
    \item $A$ is symmetrisable (modulo $p$)
    \item The cyclic products (modulo $p$) of $A$ and of $A^T$ are the same. 
    \item $A$ is equivalent to $A^T$ (modulo $p$)
    \item The Cartan representation associated to $A$ is isomorphic to that associated to the dual of the Kac-Moody algebra $\mathfrak{g}(A)$.
    \item There exists a $W$-invariant bilinear form on the Cartan representation, which is symmetric and unique up to multiplication by a scalar. 
    \end{itemize}
\end{theorem}

\begin{theorem}Let $A$ be a symmetrisable indecomposable generalised Cartan matrix with associated Weyl group $W$. Then the Cartan representation over $\mathbb{R}$ and the reflection representation of $W$ are isomorphic if and only if $A_{ij} A_{ji} \in \left\{0,1,2,3,4\right\}$ for all $i, j$.\end{theorem}

\begin{theorem}Let $A$ be a symmetrisable indecomposable generalised Cartan matrix with associated Weyl group $W$. Then the following are equivalent: 

\begin{itemize}
\item $\det A \neq 0$.
\item The $W$-invariant bilinear form on the Cartan representation over a field of characteristic zero is non-degenerate. 
\item The Cartan representation is isomorphic to its dual. 
\end{itemize}
\end{theorem}

\begin{theorem}Let $A$ be a generalised Cartan matrix which is symmetrisable and indecomposable modulo $p$ for some odd prime $p$, and $W$ its Weyl group. Then the following are equivalent: 
\begin{itemize}
    \item $p$ does not divide $\det A$.
    \item There exists a non-degenerate $W$-invariant bilinear form on the Cartan representation over a field of characteristic $p$.
    \item The Cartan representation is isomorphic to its dual.
\end{itemize}
\end{theorem}

There are several ways in which this dissertation could have been extended, but time did not allow. Firstly to try and extend the results for finite characteristic to those GCMs which cannot be broken down into indecomposable components. It is likely that a condition on cycles which generalises the cyclic product condition would be necessary.

Secondly, to explore the reflection representation in other fields. This is a problem which is outside of the scope of the text, but we'll note here that it is possible to generalise the definition. The difficulty is that it requires a notion of $2 \cos \theta$ for certain angles $\theta = \frac{\pi}{n}$. But we can write down $2 \cos \theta = e^{i \theta} + e^{-i \theta}$ over a field $k$ of characteristic $0$.

If we take a Coxeter group $W$ generated by $s_1, \ldots, s_k$ such that $n$ is the lowest common multiple of the orders of the products $s_is_j$, then the reflection representation will be well-defined over any extension of $\mathbb{Q}$ which contains $e^{ \frac{2 \pi i}{n}}$, and in fact over the ring of integers $\mathcal{O}(\mathbb{Q}(e^{\frac{2 \pi i}{n}}))$. 

And we can further reduce modulo a prime $p \in \mathbb{Z}$ to get a notion of the reflection representation in characteristic $p$ by considering the reduction by a prime ideal $\mathfrak{p}$ in this ring of integers which lies over $p$.

Thirdly, it would be possible to further integrate this text with the paper of Vinberg (\cite{Vinberg}). This studies more general groups and representations than those explored here, and takes a largely geometric approach, but it briefly states a result that generalises Theorem \ref{thm:isomorphcond} and gives several equivalent conditions relating to symmetrisability.

\bibliographystyle{plainnat}

% Start your bibliography here

% with a bib file:

% or by listing the items
\bibliography{Bibliography}

%\begin{thesisauthorvita}             %% Write your vita here; it can be
%                                     %% anything in LaTeX2e par-mode.
%\end{thesisauthorvita}               %%

\end{document}